\documentclass[11pt,reqno]{amsart}
\usepackage{amsmath,amssymb,amsfonts}
\usepackage[utf8]{inputenc}
\usepackage{enumerate}
\usepackage[paperwidth=8.26in,paperheight=11.69in,
left=1.125in, right=1.125in, bottom=1.25in]{geometry}
\usepackage{thmtools}
\declaretheoremstyle[
 headfont=\normalfont\bfseries,
 headindent= 0pt,
 bodyfont=\em,
 spaceabove=8pt,
 spacebelow=8pt
]{thm}
\declaretheoremstyle[
 headfont=\normalfont\em,
 headindent= 0pt,
 spaceabove=8pt,
 spacebelow=8pt
]{remark}
\declaretheoremstyle[
 headfont=\normalfont\bfseries,
 headindent= 0pt,
 spaceabove=8pt,
 spacebelow=8pt
]{example}
\declaretheoremstyle[
 headfont=\normalfont\bfseries,
 headindent= 0pt,
 spaceabove=8pt,
 spacebelow=8pt
]{definition}
\declaretheorem[name=Theorem,style=thm,numberwithin=section,
]{thm}
\declaretheorem[name=Proposition,style=thm,sibling=thm]{prop}
\declaretheorem[name=Lemma,style=thm,sibling=thm]{lem}
\declaretheorem[name=Corollary,style=thm,sibling=thm]{cor}
\declaretheorem[name=Example,style=example,sibling=thm]{example}
\declaretheorem[name=Definition,style=thm,sibling=thm]{defn}
\declaretheorem[name=Remark,style=remark]{rem}

\usepackage{cleveref}
\crefname{thm}{Theorem}{Theorems}
\crefname{prop}{Proposition}{Propositions}
\crefname{lem}{Lemma}{Lemmas}
\crefname{cor}{Corollary}{Corollaries}
\crefname{example}{Example}{Examples}
\crefname{defn}{Definition}{Definitions}
\crefname{rem}{Remark}{Remarks}

\crefname{enumi}{}{}
\crefname{enumii}{}{}
\crefname{enumiii}{}{}
\crefname{equation}{}{}

\numberwithin{equation}{section}

\newcommand{\Wm}{\mathcal{W}}

\newcommand{\aba}{\bar{\alpha}}
\newcommand{\bba}{\bar{\beta}}

\newcommand{\sba}{\bar{\sigma}}
\newcommand{\wba}{\bar{w}}
\newcommand{\Wba}{\overline{W}}
\newcommand{\zba}{\bar{z}}
\newcommand{\Zba}{\overline{Z}}
\newcommand{\mean}{H}

\newcommand{\jba}{\bar{j}}
\newcommand{\kba}{\bar{k}}

\DeclareMathOperator{\Ric}{Ric}

\DeclareMathOperator{\vol}{Vol}

\DeclareMathOperator{\tf}{tf}

\renewcommand{\Re}{\operatorname{Re}}

\newcommand{\II}{I\!I}
\begin{document}
\title[Semi-isometric CR immersions of CR manifolds]{Semi-isometric CR immersions of CR manifolds into Kähler manifolds and applications}
\author{Duong Ngoc Son}
\address{Fakultät für Mathematik, Universität Wien, Oskar-Morgenstern-Platz 1, 1090 Wien, Austria }
\email{son.duong@univie.ac.at}

\date{March 7, 2020}
\thanks{The author was supported by the Austrian Science Fund, FWF-Projekt M 2472-N35.}
\begin{abstract}
	We study the second fundamental form of semi-isometric CR immersions from strictly pseudoconvex CR manifolds into K\"ahler manifolds. As an application, we give a precise condition for the CR umbilicality of real hypersurfaces, extending an well-known theorem by Webster on the nonexistence of CR umbilical points on generic real ellipsoids. As other applications, we extend the linearity theorem of Ji-Yuan for CR immersions into spheres with vanishing second fundamental form to the important case of three-dimensional manifolds, and prove the ``first gap'' theorem in the spirit of Webster, Faran, Cima-Suffridge, and Huang for semi-isometric CR immersions into a complex euclidean space of ``low'' codimension. Our new approach to the linearity theorem is based on the study of the first positive eigenvalue of the Kohn Laplacian.
\end{abstract}
\maketitle

\section{Introduction}
Let $\iota \colon M \hookrightarrow \mathbb{C}^{n+1}$ be a strictly pseudoconvex CR manifold and $\rho$ a strictly plurisubharmonic defining function for $M$, i.e., $M = \{Z\in \mathbb{C}^{n+1} \colon \rho(Z) = 0\}$ and $d\rho \ne 0$ on $M$. Then $\rho$ induces a Kähler metric $\omega:= i\partial\bar{\partial}\rho$ in a neighborhood of $M$ in $\mathbb{C}^{n+1}$ and a pseudohermitian structure $\theta: = \iota^{\ast}(i\bar{\partial}\rho)$ on $M$. The Kähler geometry of $\omega$ and the pseudohermitian geometry of $\theta$ have interesting relations, as exploited implicitly in, for examples,  \cite{webster1978pseudo,li--luk,li--son,li--lin--son}. In this paper, we consider the following generalizations: Let $(\mathcal{X},\omega)$ be a complex hermitian manifold with the fundamental $(1,1)$-form $\omega$. A smooth CR immersion $F\colon (M,\theta)\to (\mathcal{X},\omega)$ is said to be \textit{semi-isometric} if 
\begin{equation}
	d\theta = F^{\ast} \omega.
\end{equation}
In this case, we identify $M$ locally with its image $F(M)\subset \mathcal{X}$ and consider the inclusion $\iota \colon F(M) \hookrightarrow \mathcal{X}$, and say that $(M,\theta)$ is a \textit{pseudohermitian submanifold} of $(\mathcal{X},\omega)$. We identify $\mathbb{C}TM$ as a subspace of $\mathbb{C}T\mathcal{X}$ in the natural way and define the \textit{pseudohermitian second fundamental form} $\II$ by using the Chern and Tanaka-Webster connections on $\mathcal{X}$ and $M$, respectively. Precisely, if $Z$ and $W$ are two vectors tangent to $M$ which extend smoothly to a neighborhood of a point $p\in M$ in $\mathcal{X}$, then $\II$ is defined by the following Gauß formula:
\begin{equation}\label{e:gaussform}
\II(Z,W) 
:=
\widetilde{\nabla}_ZW - \nabla_ZW,
\end{equation}
where $\widetilde{\nabla}$ and $\nabla$ are the Chern and Tanaka-Webster connections on $\mathbb{C}T\mathcal{X}$ and $\mathbb{C}TM$, respectively. Observe that if $W\in \Gamma(T^{1,0}M)$, then $\II(Z,W)$ is a section of $T^{1,0}\mathcal{X}$ along $M$. Therefore, we define the $(1,0)$-\textit{mean curvature vector} $H$ to be the $(1,0)$-field along $M$ given by
\begin{equation}
	\mean
	:=
	\frac{1}{n} \sum_{\alpha=1}^n \II(Z_{\aba}, Z_{\alpha}),
\end{equation}
where $\{Z_{\alpha} \colon \alpha = 1,2,\dots , n\}$ is an orthonormal basis for $T^{1,0}M$ (with respect to the Levi-metric $-id\theta$) and $Z_{\aba}:= \Zba_{\alpha}$ is the conjugate basis. In analogy with the notion of the mean curvature for Riemannian immersions, we call $|\mean|$ the \textit{mean curvature function} of~$M$ in~$\mathcal{X}$.

The first purpose of this paper is to show that the squared mean curvature function $|H|^2$ agrees with the so-called transverse curvature of \cite{graham1988smooth} when $(M,\theta)$ is defined by an appropriate function. Therefore, by a recent result of Li and the author \cite{li--son}, $n$ times the average value of $|H|^2$ gives an upper bound for the first positive eigenvalue $\lambda_1$ of the Kohn Laplacian~$\Box_b$ in the case $(\mathcal{X},\omega)$ is the complex euclidean space and $M$ is compact. Recall that if $(M,\theta)$ is a compact strictly pseudoconvex embeddable CR manifold, then $\Box_b: = \bar{\partial}_b^{\ast} \bar{\partial}_b$ acting on functions is a nonnegative self-adjoint operator on $L^2(M,d\vol_{\theta})$, where $d\vol_{\theta} := \theta\wedge (d\theta)^n$. The spectrum of $\Box_b$ consists of $0$ and positive eigenvalues $\lambda_1 < \lambda_2 < \cdots < \lambda_k < \cdots \to \infty$, each has finite multiplicity \cite{beals1988calculus,burns--epstein}. As mentioned above, the ``Reilly-type'' bound for $\lambda_1$ of \cite{li--son} can be reformulated as follows.
\begin{thm}[Li-Son \cite{li--son}]\label{thm:1}
	Let $(M^{2n+1},\theta)$ be a compact strictly pseudoconvex pseudohermitian manifold, $F \colon M\to \mathbb{C}^N$ a semi-isometric CR immersion, and $\lambda_1$ the first positive eigenvalue of the Kohn Laplacian. Then
	\begin{equation}\label{e:est0}
	\lambda_1 \leq \frac{n}{\vol(M)} \int_M \left(|\mean_{F(M)}|^2 \circ F\right) \,d\vol_{\theta}.
	\end{equation}
	If the equality holds, then each $b^I: = \Box_b \overline{F}^I$, $I = 1,2,\dots, N$, is either a constant or an eigenfunction that corresponds to $\lambda_1$. 
\end{thm}
This theorem is a main motivation for this paper and plays a crucial role in the proof of the linearity result below (\cref{thm:2}). As already mentioned, we shall prove in \cref{sec:transverse} that $|\mean|^2$ and the transverse curvature $r(\rho)$ of a defining function $\rho$ coincide if $\rho$ is chosen appropriately and hence \cref{e:est0} follows from Theorem~1.1 in \cite{li--son}. We shall also show that \cref{e:est0} follows from a bit more general estimate for $\lambda_1$ in terms of the ``pseudohermitian total tension'' of a $\mathcal{C}^2$ map into a K\"ahler manifold; see \cref{sec:tension} for precise definitions.

The second purpose of this paper is to study some natural questions that arise when considering semi-isometric immersions into a complex euclidean space. Our answers to these questions generalize some well-known results for CR immersions into the unit sphere of $\mathbb{C}^N$ in \cite{webster1979rigidity,faran1986linearity,huang1999linearity,ji2010flatness}. In particular, we shall address the questions of the vanishing of $\II$ restricted to the holomorphic tangent space $T^{1,0}M \oplus T^{0,1}M$, or the vanishing of its traceless component~$\II^{\circ}$ ($\II^{\circ}$ vanishes iff $\II(Z,W) = 0$ for all $Z,W\in T^{1,0}M$). It turns out that the trace of $\II$ (with respect to the Levi metric) never vanishes and therefore we ask what if $\II^{\circ} = 0$? If $\II^{\circ}$ vanishes at $p$, then we say, in analogy with classical surface theory, that $p$ is a \textit{pseudohermitian umbilical} point of~$F$. If $F$ is pseudohermitian umbilical at every points, we say that $F$ is a \textit{totally pseudohermitian umbilic} immersion. \cref{thm:2} below settles the question about totally umbilicity in the case where the ambient space is the euclidean space.

\begin{thm}\label{thm:2} Let $F \colon (M^{2n+1},\theta) \hookrightarrow (\mathbb{C}^N,\omega)$, $\omega:=i\partial\bar{\partial}\|Z\|^2$, be a semi-isometric CR immersion into a complex euclidean space. Suppose $(M,\theta)$ is complete and $\II (Z,W) = 0$ for any $(1,0)$-vectors tangent to $M$, then $(M,\theta)$ is globally CR equivalent to a sphere $\mathbb{S}^{2n+1} \subset \mathbb{C}^{n+1}$ and there exists a CR diffeomorphism $\varphi \colon \mathbb{S}^{2n+1} \to M^{2n+1}$ such that $F \circ \varphi$ extends to a linear mapping between complex spaces.
\end{thm}

Notice that the conclusion of this theorem also says that $F$ maps $M^{2n+1}$ into a sphere. We shall say that $F$ ``\textit{realizes an immersion in a sphere}'' if there exists a CR immersion $\phi \colon M \to \mathbb{S}^{2N-1}$ such that $F = \iota \circ \phi$, where $\iota\colon \mathbb{S}^{2N-1} \hookrightarrow \mathbb{C}^N$ is the standard inclusion of the sphere. This terminology is borrowed from Takahashi \cite{takahashi1966minimal}. Let $\theta: = \phi^{\ast}\Theta$, where $\Theta: = \iota^{\ast}(i\bar{\partial} \|Z\|^2)$ is the standard pseudohermitian structure on the sphere. Then $F$ is a semi-isometric immersion from $(M,\theta)$ into $(\mathbb{C}^N, i\partial\bar{\partial} \|Z\|^2)$. We shall prove in \cref{prop:2sff} that the CR second fundamental form $\II_M^{CR}$ of $\phi$ and the traceless component $\II^{\circ}$ of $F$ are essential the same. Therefore, the linearity of $F$ in this particular case also follows from the result of Ji-Yuan \cite{ji2010flatness} who exploited an useful normalization technique of Huang~\cite{huang1999linearity}. We expect this normalization technique extends to the case of semi-isometric immersions without the assumption that they realize immersions in a sphere. However, we shall not go in this direction, but approach to the linearity in \cref{thm:2} along a different route. We first show that if the immersion is totally umbilic, then $(M,\theta)$ is ``extremal'' for the lower and upper estimates of the first positive eigenvalue $\lambda_1$ of the Kohn Laplacian of \cite{chanillo--chiu--yang} (cf. \cite{li--son--wang}) and \cite{li--son} (i.e., \cref{thm:1} above). This allows us to conclude that $(M,\theta)$ is globally CR equivalent to the sphere by applying the main result of \cite{li--son--wang}. We then exploit the fact that the first eigenfunctions of the Kohn Laplacian on the standard sphere are the restrictions of the homogeneous harmonic polynomials of bi-degree $(0,1)$ to deduce that $F$ becomes linear after being pre-composed with an automorphism of the source. This concludes the proof of \cref{thm:2}.

It is worth pointing out that \cref{thm:2} covers the \textit{three-dimensional} case. This interesting case is more difficult since the pseudoconformal Gauß equation of \cite{ebenfelt2004rigidity} does not give any useful information: The Chern-Moser tensor vanishes trivially in this dimension. Moreover, a certain Bianchi identity relating the covariant derivatives of the pseudohermitian Ricci and torsion tensors, which is useful in the higher dimensional case, is rendered trivial in three-dimension (see \cref{lem:sp}). We shall need pseudohermitian Gauß-Codazzi-Mainardi equations which relate the second fundamental form (resp. its covariant derivatives) and the tangential (resp. normal) component of the ambient curvature tensor. These equations allow us to deduce that the intrinsic scalar curvature $R$, which in our situation agrees with $2$ times the squared mean curvature function $|H|^2$, is constant. This is important for us to deduce that $(M,\theta)$ is extremal for the aforementioned eigenvalue bounds. We point out that in three-dimensional case, we can also deduce from this and the vanishing of the pseudohermitian torsion that $M$ is locally CR spherical by proving directly the vanishing of Cartan's 6th-order umbilical tensor. We therefore obtain the following
\begin{cor}\label{cor:3dim}
	Let $\phi \colon M^{2n+1} \to \mathbb{S}^{2N-1}$ be a smooth CR immersion, with $n\geq 1$ and $N\geq n+1$. If $\II_M^{CR} = 0$, then
	\begin{enumerate}[(i)]
		\item $M$ is locally CR spherical.
		\item For each $p\in M$, there exist a neighborhood $U$ of $p$ in $M$, an open set $V \subset \mathbb{S}^{2n+1}$, and a CR diffeomorphism $\gamma \colon V \to U$ such that $\phi \circ \gamma$ extends to a totally geodesic CR embedding of $\mathbb{S}^{2n+1}$ into $\mathbb{S}^{2N-1}$. 
	\end{enumerate}
\end{cor}
Here a CR immersion between spheres is totally geodesic iff it is spherical equivalent to the linear embedding. As already mentioned above, the case $\dim_{\mathbb{R}}M \geq 5$ is well-known and due to Ebenfelt-Huang-Zaitsev \cite{ebenfelt2004rigidity} and Ji-Yuan \cite{ji2010flatness} for Parts (i) and (ii), respectively.

Using an argument based on Huang's lemma (Lemma~3.2 in \cite{huang1999linearity}), as was done in Proposition~5.2 of \cite{ebenfelt2004rigidity}, we obtain from \cref{thm:2} the following generalization of the ``first gap'' theorem by Webster \cite{webster1979rigidity}, Cima-Suffridge \cite{cima1983reflection}, Faran \cite{faran1986linearity}, Huang \cite{huang1999linearity} which treat the case when $F$ is assumed to realize an immersion in a sphere.
\begin{thm}\label{cor:linearity}
	Let $F \colon (M^{2n+1},\theta) \hookrightarrow (\mathbb{C}^N,\omega)$, $\omega:=i\partial\bar{\partial}\|Z\|^2$, be a semi-isometric CR immersion into a complex euclidean space, $n\geq 2$. Suppose $(M,\theta)$ is complete, $N\leq 2n$, and $M$ is locally CR spherical, then $(M,\theta)$ is globally CR equivalent to the sphere and there exists a CR diffeomorphism $\varphi \colon \mathbb{S}^{2n+1} \to M^{2n+1}$ such that $F \circ \varphi$ extends to a linear mapping between complex spaces. In particular, $F$ realizes an immersion into a sphere in $\mathbb{C}^{N}$.
\end{thm}
It is worth pointing out that the conclusions in 
\cref{thm:2,cor:linearity} are global. Although, we do not assume any topological assumption on~$M$, we do assume that the immersion is globally defined. On the other hand, a local version of \cref{cor:linearity} for $F$ realizing an immersion in a sphere can be obtained by using the well-known fact that local (rational) holomorphic maps between connected pieces of spheres extend to global maps with poles off the source sphere. 

It is not unexpected that the codimension restriction in \cref{cor:linearity} is sharp. In fact, an well-known example in the case of maps into a sphere for the case $N=2n+1$ also serves as a counterexample for our more general situation. Precisely, the complex Whitney map $\mathcal{W}$ from $\mathbb{S}^{2n+1}$ to $\mathbb{S}^{4n+1}$ induces a semi-isometric immersion from $(\mathbb{S}^{2n+1}, \mathcal{W}^{\ast} (i\bar{\partial}\|Z\|^2))$ into $\mathbb{C}^{2n+1}$, but $\theta: = \mathcal{W}^{\ast} (i\bar{\partial}\|Z\|^2)$ is \textit{not} homothetic to the standard pseudohermitian structure on the source sphere; see \cref{ex:whitney} for more details.

Another interesting question, going back to the seminal paper of Chern and Moser \cite{chern1974real}, that we are able to tackle with our current techniques is the existence of the CR umbilical points on (Levi-nondegenerate) CR manifolds. This problem has been studied by Webster \cite{webster2000holomorphic} for the case $n\geq 2$ and by, for example, Huang-Ji \cite{huang2007every}, Ebenfelt et al. \cite{ebenfelt2017umbilical,ebenfelt2018family} for the case $n=1$. Recall that if $n\geq 2$ (i.e., $\dim_{\mathbb{R}}M \geq 5$), then $p\in M$ is a CR umbilical point iff the Chern-Moser tensor of $M$ vanishes at $p$ \cite{chern1974real}. This notion of (intrinsic) CR umbilical points and that of (extrinsic) pseudohermitian umbilical points of an immersion into a complex euclidean space are closely related. This close relation was already noticed and exploited in \S 5 of \cite{ebenfelt2004rigidity} for the case when $F$ realizes an immersion in a sphere. In fact, these two properties are equivalent for the immersions having ``low'' codimension (\cref{prop:2um}). This equivalence allows us to locate the CR umbilical points on a strictly pseudoconvex CR manifold $M$ when it admits a pseudohermitian structure $\theta$ for which $(M,\theta)$ is semi-isometrically immersed into a complex euclidean space of low codimension. Precisely, assume that $F=(F_1,\dots,F_N)$ is a holomorphic map from an open set in $\mathbb{C}^{n+1}$ into $\mathbb{C}^N$ and $\rho = \|F\|^2 + \psi$, where $\|F\|^2:=\sum_{d=1}^N |F^d|^2$ is the ``squared norm'' of~$F$ and $\psi$ is pluriharmonic. If $M:=\{\rho = 0\}$ is a strictly pseudoconvex real hypersurface and $\theta: = i\bar{\partial}\rho$, then $(M,\theta)$ is semi-isometrically immersed into $\mathbb{C}^N$ by~$F$. The next result stated in this introduction is a criterion for the CR umbilicity which is formulated as a property of the (Levi-) Fefferman determinant~$J(\rho)$. Recall that $J(\rho)$ is defined by
\begin{equation}\label{e:fm}
	J(\rho) = - \det \begin{bmatrix}
	\rho & \rho_{\kba} \\
	\rho_{j} & \rho_{j\kba}
\end{bmatrix},
\end{equation}
where $\rho_j = \partial\rho/\partial z^j$ and $\rho_{j\kba} = \partial^2\rho/\partial z^j \bar{z}^k$.
\begin{thm}\label{thm:umbilichypersurface} Let $\iota \colon M^{2n+1}\subset \mathbb{C}^{n+1}$ be a strictly pseudoconvex real hypersurface defined by $\rho = 0$ with $d\rho \ne 0$ and $J(\rho) > 0$ along $M$, $n\geq 2$. Suppose that $\rho = \|F\|^2 + \psi$ where $F$ is a holomorphic map into $\mathbb{C}^N$ and $\psi$ is pluriharmonic. Then,
\begin{equation}\label{e:umbilic}
	\iota^{\ast} (i\partial\bar{\partial} \log J(\rho))|_{H(M)} \geq 0.
\end{equation}
If the equality occurs at $p\in M$, then $p$ is a CR umbilical point of $M$. If in addition $N \leq 2n$, then the equality occurs if $p$ is CR umbilical. In particular, if the complex hessian of $\log J(\rho)$ has at least two nonzero eigenvalues at every points and $N\leq 2n$, then $M$ admits no CR umbilical points.
\end{thm}
Perhaps the most interesting nontrivial example for which \cref{thm:umbilichypersurface} applies is that of real ellipsoids. This example was treated in \cite{webster2000holomorphic} which studies the complete integrability of the Reeb flow associated to the ``normalized'' contact form (the one for which the Chern-Moser tensor has unit norm). Precisely, let $A = (A_1,A_2,\dots , A_{n+1})$ be a set of real numbers. The real ellipsoid $E(A)$ is the strictly pseudoconvex real hypersurface defined by $\rho = 0$, where
\begin{equation}\label{e:elipdef}
	\rho: = \|Z\|^2 + \Re \sum_{j=1}^{n+1} A_j z_j^2 - 1.
\end{equation}
We obtain the following corollary which was first proved in Theorem~0.1 of Webster \cite{webster2000holomorphic} (the original statement in \cite{webster2000holomorphic} is for ``generic'' ellipsoids that satisfy $0 < A_1 < A_2 < \cdots < A_{n+1} < 1$).
\begin{cor}[Webster \cite{webster2000holomorphic}]\label{cor:webster}
	A real ellipsoid $E(A)$ in $\mathbb{C}^N$, with $N\geq 3$, admits no CR umbilical points, provided that there are at least 2 nonzero components in $A$.
\end{cor}
If $A$ has exactly one nonzero component, then using \cref{e:umbilic} we can locate precisely the nonempty CR umbilical locus of $E(A)$; see \cref{rem:elip}.

The case $\dim_{\mathbb{R}} M = 3$ (i.e., $n=1$) is fundamentally different as the CR umbilical property is not characterized by the Chern-Moser tensor but the Cartan's 6th-order tensor (which does not appear in the Gauß equation \cref{e:gauss}). In fact, compact real ellipsoids in $\mathbb{C}^2$ always admit umbilical points \cite{huang2007every} while an unbounded ellipsoidal tube (when $A_j = 1$ for all $j$ in \cref{e:elipdef}) admits no umbilical points \cite{ebenfelt2018family}; see also \cite{ebenfelt2017umbilical} for further results in the three-dimensional case.

The paper is organized as follows. In \cref{sec:2} we quickly recall some background in pseudohermitian geometry and study the notion of second fundamental form for the semi-isometric CR immersions of CR manifolds into Kähler manifolds. Precisely, we prove the Gauß-Codazzi-Mainardi equations and establish the relation between the mean curvature and the so-called transverse curvature. In \cref{sec:um}, we study the relations between the two notions of umbilical points and prove \cref{thm:umbilichypersurface}. In \cref{sec:bel}, we give a Beltrami-type formula for the Kohn Laplacian which we need for the study of eigenvalue estimates and prove a Takahashi-type theorem. In \cref{sec:tension}, we prove a simple upper bound for the first positive eigenvalue of the Kohn Laplacian on a CR manifold in terms of the total tension and $\bar{\partial}_b$-energy of a map into a Kähler manifold and prove \cref{thm:1}. In \cref{sec:6}, we prove \cref{thm:2,cor:linearity,cor:3dim}. In the last section, we give an example illustrating the necessity of the certain codimension restriction imposed at various places in the paper.

\section{Semi-isometric CR immersions and the Gauß-Codazzi-Mainardi equations}\label{sec:2}
\subsection{Pseudohermitian geometry}
For readers' convenience, we quickly recall some notions and facts about the pseudohermitian geometry of CR manifolds. We refer to \cite{tanaka1975differential,webster1978pseudo,dragomir--tomassini} for more details.
Let $(M^{2n+1},T^{0,1}M)$ be a strictly pseudoconvex CR manifold of hypersurface type, i.e., $\dim_{CR}M = n$. There exists a real contact 1-form $\theta$ such that the holomorphic tangent space $H:=\Re (T^{1,0}M \oplus T^{0,1}M) $ is given by the kernel of $\theta$ (i.e., $H= \ker \theta$) and the two-form $d\theta$ is positive definite on $H(M)$. The pair $(M^{2n+1},\theta)$ is called a pseudohermitian manifold by Webster \cite{webster1978pseudo}. The Reeb field associated to $\theta$ is the unique real vector field $T$ satisfying $T \rfloor d\theta = 0$ and $\theta(T) = 1$. 
The Tanaka-Webster connection on $M$ is the unique affine connection $\nabla \colon \Gamma(TM) \to \Gamma(T(M) \otimes T(M)^{\ast})$ for which the complex structure $J$, the contact structure $H(M)$, and the Reeb field $T$ are parallel, and its torsion is pure~\cite{dragomir--tomassini}. Here the torsion $\mathbb{T}_{\nabla}$ is said to be pure if (see \cite{dragomir--tomassini})
\begin{equation}
 \mathbb{T}_{\nabla}(X,Y) = d\theta(X,Y) T,
\end{equation}
and
\begin{equation}
 \mathbb{T}_{\nabla}(T,JY) = - J\mathbb{T}_{\nabla}(T, Y)
\end{equation}
for all $X,Y \in H(M)$. See Proposition 3.1 of \cite{tanaka1975differential} or Theorem~2.1 of \cite{webster1978pseudo} or \cite{dragomir--tomassini} for a proof. We shall identify $\nabla$ with its complexified connection on $\mathbb{C}TM$ as usual.
The pseudohermitian structure $\theta$ also induces a hermitian metric on $H(M)$ by
\[
	G_{\theta}(X,Y) = d\theta(X, JY),
\]
which extends to $\mathbb{C}H(M)$ by complex linearity. The adapted Riemannian metric $g_{\theta}: = G_{\theta} + \theta^2$ agrees with $G_{\theta}$ when restricted to $H(M)$. We say that $(M,\theta)$ is \textit{complete} if $g_{\theta}$ is a complete metric. 
\subsection{The second fundamental form and the $(1,0)$-mean curvature vector}\label{sec:mean}
\begin{defn} Let $(M,\theta)$ be a strictly pseudoconvex pseudohermitian manifold, $(\mathcal{X},\omega)$ a complex Hermitian manifold, and $F\colon M\to (\mathcal{X},\omega)$ a smooth CR mapping. We say that $F$ is \textit{semi-isometric} if
\begin{equation}\label{e:semi-isometric}
	d\theta = F^{\ast} \omega.
\end{equation}
\end{defn}
It seems to be more natural to require that $d\theta$ agrees with $F^{\ast}\omega$ when restricted to $H(M)$. However, when $ \dim_{C\!R}M \geq 2$ and $\omega$ is K\"ahler, this seemingly weaker condition is actually equivalent to \cref{e:semi-isometric}.
\begin{prop}
	Let $F \colon (M,\theta) \to (\mathcal{X}, \omega)$ be a CR mapping. Assume that $\omega$ is Kähler and $M$ has dimension at least 5, then $F$ is semi-isometric iff
	\begin{equation}\label{e:weaker}
		d\theta|_{H(M)} = F^{\ast} \omega|_{H(M)}.
	\end{equation}
\end{prop}
\begin{proof}
	Assume that \cref{e:weaker} holds. The restriction of the closed two form $\eta: = d\theta - F^{\ast} \omega$ to $H(M)$ vanishes. Thus, by Lemma~3.2 of \cite{lee1988pseudo}, it must vanish identically.
\end{proof}

It is worth pointing out that, for semi-isometric immersions, the adapted Riemannian metric $g_\theta: = G_{\theta} + \theta^2$ does not coincide with the induced Riemannian metric from the ambient space. In fact, $g_{\theta}(T,T) = 1$, but $\langle T, T\rangle_{\omega}$ equals $2$ times the mean curvature function and is not constant in general.

In local computations, we can suppose that $M \subset \mathcal{X}$, $T^{1,0} M= \mathbb{C}TM \cap T^{1,0} \mathcal{X}$, and $F$ is the inclusion. In this case, we shall say that $(M,\theta)$ is a \textit{pseudohermitian submanifold} of $(\mathcal{X},\omega)$. We denote by $\widetilde{\nabla}$ the Chern connection on $\mathcal{X}$ and by $\nabla$ the Tanaka-Webster connection on~$M$. For any sections $U, V \in \Gamma(\mathbb{C}TM)$, we extend $V$ smoothly to a section $\widetilde{V}$ of $\mathbb{C}T\mathcal{X}$ and observe that $\widetilde{\nabla}_U \widetilde{V}$ does not depend on the extension.
\begin{defn} Let $\iota \colon (M,\theta) \hookrightarrow (\mathcal{X},\omega)$ be a pseudohermitian submanifold of a Hermitian manifold $(\mathcal{X},\omega)$. The second fundamental form of $M = \iota(M)\subset \mathcal{X}$ is defined to be
\begin{equation}\label{e:2.5a}
	\II (U,V) 
	=
	\widetilde{\nabla}_U\widetilde{V} - \nabla_UV.
\end{equation}
\end{defn}
We define the normal subbundle $N^{1,0}M$ (resp. $N^{0,1}M$) to be the orthogonal complement of $T^{1,0}M$ (resp. $T^{0,1}M$) in $T^{1,0}\mathcal{X}$ (resp. $T^{0,1}\mathcal{X}$) with respect to the Hermitian metric on~$\mathcal{X}$. Basic properties of $\II$ are as follows.
\begin{prop}\label{prop:basicsff} Let $\iota \colon (M,\theta) \hookrightarrow (\mathcal{X},\omega)$ be a pseudohermitian submanifold of a Hermitian manifold $(\mathcal{X},\omega)$. Then the second fundamental form $\II$ is well-defined and tensorial. If, moreover, $(\mathcal{X},\omega)$ is Kähler, then for any $Z,W \in \Gamma(T^{1,0}M \oplus T^{0,1}M)$
	\begin{align}\label{e:reality}
		\II(\Wba , Z) &= \overline{\II(W , \Zba)}, \\
		\II(Z, W) & = \II(W , Z) - i \langle Z, W\rangle T, \label{e:peter} \\
		\II(T,Z) & = \II(Z,T) - \tau Z. \label{e:iizt}
	\end{align}
	Here $\tau Z := \mathbb{T}_{\nabla}(T,Z)$ is the pseudohermitian torsion. Furthermore, $\II(Z,W)$ takes values in $N^{1,0}M\oplus N^{0,1}M$.
\end{prop}
\begin{proof}	
	Equation \cref{e:reality} follows from the reality of the Chern and Tanaka-Webster connections. Extend $Z$ and $W$ smoothly to $\mathcal{X}$. Since the Chern connection $\widetilde{\nabla}$ has no torsion, 
	\begin{equation}
		\widetilde{\nabla}_{Z} W - \widetilde{\nabla}_{W} Z = [Z , W].
	\end{equation}
	On the other hand, the Tanaka--Webster torsion is pure, i.e.,
	\begin{equation}
	\nabla_{Z} W - \nabla _{W} Z = [Z , W] + i\langle Z ,W \rangle T,
	\end{equation}
	where $T$ is the Reeb vector field. Therefore,
	\begin{equation}
	\II(Z, W) - \II(W , Z) = -i\langle Z ,W \rangle T.
	\end{equation}
	This proves \cref{e:peter}.
	
	Extend $Z$ and $T$ to smooth vector fields on a neighborhood of a point $p\in M$ in $\mathcal{X}$. Observe that 
	\begin{align}
	\widetilde{\nabla}_TZ - \widetilde{\nabla}_ZT = [T,Z].
	\end{align}
	Moreover,
	\begin{align}
	\nabla_TZ - \nabla_ZT = [T,Z] + \mathbb{T}_{\nabla}(T,Z) = [T,Z] + \tau Z.
	\end{align}
	Subtracting these two identities, we obtain \cref{e:iizt}.

	To prove the last statement, first we consider the case $\Wba$ is a $(0,1)$-vector field, then $\II (Z,\Wba)$ is a section of $T^{0,1}\mathcal{X}$ along $M$. Moreover, for $Y \in \Gamma(T^{1,0}M)$, 
\begin{equation}
	\langle \II (Z,\Wba) , Y \rangle
	=
	\langle \II (Z,\Wba) - \II (\Wba, Z) , Y \rangle
	= -i \langle Z , \Wba \rangle \langle T , Y \rangle
	=0.
\end{equation} 
	This shows that $\II (Z,\Wba) \in \Gamma(N^{0,1}M)$. 
	
	Next we consider the case $W$ is a (1,0)-vector field, then $\II(Z,W)$ is of type $(1,0)$. Moreover, for any (0,1)-vector field $\overline{Y}$ tangent to $M$,
\begin{equation}
	\langle \II(Z,W) , \overline{Y}\rangle
	=
	-\langle Z , \II(W, \overline{Y})\rangle 
	=0,
\end{equation}
	and hence $\II (Z,W) \in \Gamma(N^{1,0}M)$, as desired. The proof is complete.
\end{proof}
Thus, the second fundamental form $\II$ is \textit{not} symmetric. However, if $Z$ and $W$ are both of type $(1,0)$, then 
\begin{equation}
	\II(Z,W) = \II(W,Z), \quad \II(\Zba,\Wba) = \II(\Wba,\Zba).
\end{equation}
Moreover, if $(M,\theta)$ has vanishing pseudohermitian torsion, then $\II(Z,T) = \II(T,Z)$.
\begin{defn}
	Let $(M,\theta) \hookrightarrow (\mathcal{X},\omega)$ be a pseudohermitian submanifold.
	The $(1,0)$-\emph{mean curvature} vector at $p$ is defined by
	\begin{equation}
		\mean (p) %= \frac{1}{n}\trace \II(p)
		=
		\frac{1}{n}\sum_{\alpha=1}^{n} \II(Z_{\aba} , Z_{\alpha}).
	\end{equation}
	Here $\{Z_{\alpha}\colon \alpha = 1,2,\dots, n\}$ is an orthonormal basis for $T^{1,0}M$ and $Z_{\aba}:= \overline{Z_{\alpha}}$.
\end{defn}
Thus, $\mean$ is a section of $T^{1,0}\mathcal{X}$ along $M$. The mean curvature at $p$ is defined to be 
\begin{equation}
\mu(p):= |\mean (p)|_{\omega},
\end{equation} 
These definitions are similar to those of the scalar and vector mean curvatures of Riemannian submanifolds.
\begin{prop}
	Let $M\hookrightarrow \mathcal{X}$ be a pseudohermitian submanifold of a K\"ahler manifold $(\mathcal{X},\omega)$. If $T$ is the Reeb field of $\theta$, then
	\begin{equation}\label{e:meanreeb}
		\mean - \overline{\mean} = iT, 
	\end{equation}
	and
	\begin{equation}\label{e:sffbu}
		\II(Z, \Wba)
		=
		\langle Z, \Wba \rangle \overline{H},
		\quad
		\II(\Wba , Z) = \langle Z , \Wba \rangle H.
	\end{equation}
	In particular, $T$ determines the $(1,0)$-mean curvature vector field.
\end{prop}
\begin{proof}
	By \cref{e:peter}, for any $\alpha$, 
	\begin{align}
		\overline{\II(Z_{\aba} , Z_{\alpha})}
			&=
		\II(Z_{\alpha} , Z_{\aba}) \notag \\
			&=
		\II (Z_{\aba} , Z_{\alpha}) - i\langle Z_{\alpha} , Z_{\aba} \rangle\, T.
	\end{align}
	Summing over $\alpha = 1, \dots , n$, we obtain \cref{e:meanreeb}.
	
	Observe that 
	\begin{align}
		\II(Z,\Wba) - \II(\Wba, Z) 
		& = -i \langle Z, \Wba \rangle T \notag \\
		& = \langle Z, \Wba \rangle (\overline{H} - H).
	\end{align}
	Taking the (1,0) and (0,1) parts, we obtain \cref{e:sffbu}. The proof is complete.
\end{proof}
\begin{prop}
	Let $(M,\theta) \hookrightarrow (\mathcal{X},\omega)$ be a pseudohermitian submanifold of a K\"ahler manifold. If $T$ is the Reeb field of $\theta$, then for $Z \in T^{1,0}M$,
	\begin{align}\label{e:iitz1}
	\II(T,Z)
	& =
	- i \widetilde{\nabla}_Z H, \\
	\tau Z 
	& = i \widetilde{\nabla}_Z \overline{H}.\label{e:tauz}
	\end{align}
	In particular, $\widetilde{\nabla}_Z \overline{H}$ is tangent to $M$.
\end{prop}
\begin{proof}
	From $\nabla T = 0$, \cref{e:meanreeb}, and \cref{e:iizt}, we have
	\begin{align}
	\II(T,Z) 
	& =
	\II(Z,T) - \tau Z \notag \\
	& =
	\widetilde{\nabla}_ZT - \tau Z \notag \\
	& =
	i \widetilde{\nabla}_Z \overline{H} - i \widetilde{\nabla}_Z H - \tau Z.
	\end{align}
	Taking (1,0) and (0,1) parts, using the fact that $\tau Z$ is of type (0,1) since $Z$ is of type (1,0), we obtain the desired identities.
\end{proof}
\subsection{The Gauß-Codazzi-Mainardi and Weingarten equations}
In this section, we shall derive CR-analogues of the classical Weingarten and Gauß--Codazzi--Mainardi equations for the semi-isometric CR immersions. CR-analogues of the Gau\ss{} equation have been used successfully in the study of CR immersions into the spheres; see, e.g., \cite{webster1979rigidity,ebenfelt2004rigidity} and the references therein. Our derivation is similar to those previous work, but we shall need to calculate explicitly some terms arising in our new situation and therefore we shall present the detailed calculation below.

\begin{prop}[Weingarten Equation]\label{prop:wein} Let $M\hookrightarrow \mathcal{X}$ be a pseudohermitian submanifold of a K\"ahler manifold. If $N$ is a section of the $N^{1,0}M \oplus N^{0,1}M$, then
	\begin{equation}
	\langle \widetilde{\nabla}_Z N , W \rangle 
	=
	-\langle N , \II(Z,W) \rangle 
	\end{equation}
	for all sections $Z,W$ of $T^{1,0}M \oplus T^{0,1}M$.
\end{prop}
\begin{proof}
The proof uses a standard argument exploiting the fact that $\widetilde{\nabla}$ is a metric connection. Precisely, since $N$ is normal to $TM$,
\begin{align}
	\langle \widetilde{\nabla}_Z N , W \rangle
	& =
		Z\langle N, W \rangle - \langle N , \widetilde{\nabla}_Z W\rangle \notag \\
	& =
		-\langle N , \widetilde{\nabla}_ZW - \nabla_Z W\rangle \notag \\
	& = -\langle N , \II(Z,W)\rangle. \notag \qedhere
\end{align}
\end{proof}
Our convention for the curvature operator of a linear connection is
\begin{equation}
	R(X,Y)Z
	=
	\nabla_X\nabla_Y Z - \nabla_Y\nabla_X Z - \nabla_{[X,Y]} Z.
\end{equation}
If $X, Y$, and $Z$ are tangent to $M$, the Gauß formula immediately implies that
\begin{align} \label{e:2curv}
	\widetilde{R}(X,Y)Z
	& =
	R(X,Y)Z + \II(X, \nabla_YZ) - \II(Y,\nabla_XZ) - \II([X,Y], Z) \notag \\
	&\qquad + \widetilde{\nabla}_X(\II (Y,Z)) - \widetilde{\nabla}_{Y}(\II(X,Z)),
\end{align}
where $\widetilde{R}$ is the curvature operator on $\mathcal{X}$. Specializing to the ``horizontal'' vector fields of appropriate types, we obtain
\begin{prop}[equations of Gauß]\label{prop:ge}
	If $\iota \colon (M, \theta) \hookrightarrow (\mathcal{X},\omega)$ is a pseudohermitian CR submanifold and $\omega$ is Kähler, then 
\begin{enumerate}
	\item for $X,Z \in \Gamma(T^{1,0}M)$ and $\overline{Y},\overline{W} \in \Gamma(T^{0,1} M) $, the following equation holds:
\begin{align}\label{e:gauss}
		\langle \widetilde{R}(X,\overline{Y}) Z, \Wba\rangle
		& =
		\langle R(X,\overline{Y}) Z, \Wba\rangle
		+
		\langle \II (X,Z) , \II (\overline{Y}, \Wba) \rangle \notag \\
		& \qquad - |\mean |^2 \left(\langle \overline{Y} , Z \rangle \langle X ,\Wba \rangle + \langle X , \overline{Y} \rangle \langle Z , \Wba \rangle \right),
\end{align}
\item for $X,Z \in \Gamma(T^{1,0}M)$, the following equation holds:
\begin{equation}\label{e:gausstorsion}
	\langle \tau X , Z \rangle 
	=
	-i \langle \II(X,Z) , \overline{\mean} \rangle.
\end{equation}
Here, $\tau X := \mathbb{T}_{\nabla} (T,X)$ is the pseudohermitian torsion of~$\theta$.
\end{enumerate}
\end{prop}
	As briefly discussed in the introduction, the Gauß equation has been extensively used in the study of the CR immersions. In particular, the traceless part of \cref{e:gauss} has been important for the study of the rigidity of CR immersions; see e.g. \cite{webster1979rigidity,ebenfelt2004rigidity,ji2010flatness} and the references therein. We point out that the trace part of \cref{e:gauss} and the equation for the torsion \cref{e:gausstorsion} are important for our proofs of \cref{thm:2} and \cref{thm:umbilichypersurface}.
\begin{proof}[Proof of \cref{prop:ge}]
	The proof of \cref{e:gauss} is similar to that of the Gauß equation for Riemannian immersions, except that the term $\langle \II ([X , \overline{Y}], Z), \Wba \rangle$ does not necessary vanish. 
	Indeed, from \cref{e:2curv} and \cref{prop:wein,prop:basicsff}, we have
	\begin{align}\label{e:tem}
		\langle \widetilde{R}(X,\overline{Y}) Z, \Wba\rangle
		& =
		\langle R(X,\overline{Y}) Z, \Wba\rangle 
		-
		\langle \II([X,\overline{Y}], Z), \Wba \rangle \notag \\
		& \quad 
		+ \langle \widetilde{\nabla}_X (\II(\overline{Y}, Z)) , \Wba \rangle 
		- \langle \widetilde{\nabla}_{\overline{Y}}(\II(X,Z)) , \Wba \rangle \notag \\
		& =
		\langle R(X,\overline{Y}) Z, \Wba\rangle 
		-
		\langle \II([X,\overline{Y}], Z), \Wba \rangle \notag \\
		& \quad 
		- \langle \II(\overline{Y}, Z) , \II(X,\Wba) \rangle 
		+ \langle \II(X,Z) , \II(\overline{Y},\Wba) \rangle \notag \\
		& =
		\langle R(X,\overline{Y}) Z, \Wba\rangle 
		+ \langle \II(X,Z) , \II(\overline{Y},\Wba) \rangle \notag \\
		& \quad 
		- \langle \overline{Y}, Z \rangle \langle X,\Wba \rangle |H|^2 
		- \langle \II([X,\overline{Y}], Z), \Wba \rangle .
	\end{align}
	Since $
		[X , \overline{Y}]
		=
		\nabla_X \overline{Y} - \nabla_{\overline{Y}} X - i \langle X , \overline{Y} \rangle T$ and $\II(\nabla_{X}{\overline{Y}}, Z)$ and $\II(\nabla_{\overline{Y}}X, Z)$ are in the normal bundle, we deduce that
	\begin{align}
		\langle \II ([X , \overline{Y}], Z), \Wba \rangle
		& =
		\langle \II (- i \langle X , \overline{Y} \rangle T,Z) , \Wba \rangle \notag \\
		& =
		- i \langle X , \overline{Y} \rangle \langle \II(T,Z), \Wba\rangle \notag \\
		& = - \langle X , \overline{Y}\rangle \langle \widetilde{\nabla}_Z H , \Wba \rangle \notag \\
		& = \langle X , \overline{Y}\rangle \langle H , \widetilde{\nabla}_Z \Wba \rangle \notag \\
		& = \langle X , \overline{Y}\rangle \langle Z , \Wba \rangle |\mean|^2.
	\end{align}
	Plugging this into \cref{e:tem}, we obtain \cref{e:gauss}.
	
	To prove \cref{e:gausstorsion}, recall that \cref{e:tauz} $ \tau X = i \widetilde{\nabla}_X \overline{H}$ and therefore,
	\begin{align}
		\langle \tau X , Z \rangle 
		& =
		i \langle \widetilde{\nabla}_X \overline{H} , Z\rangle \notag \\
		& =
		i \left(X \cdot \langle \overline{H} , Z \rangle - \langle \overline{H} , \widetilde{\nabla}_X Z \rangle \right) \notag \\ 
		& =
		-i\langle \overline{\mean}, \widetilde{\nabla}_X Z - \nabla_XZ\rangle \notag \\
		& =
		 -i\langle \overline{\mean} , \II(X,Z)\rangle.
	\end{align}
	The proof is complete.
\end{proof}

For each section $Y$ of the normal bundle $N^{1,0}M \oplus N^{0,1}M$, we define the (Weingarten) \textit{shape operator} $A_{Y}$ from $T^{1,0}M \oplus T^{0,1}M$ into itself:
\begin{align}
\langle A_{Y}Z , W \rangle 
=
\langle \II(Z,W) , Y \rangle.
\end{align}
Since $\II$ is not symmetric, $A_Y$ is not symmetric either. However, it has some nice properties. In particular,
\begin{equation}
A_YZ 
= \langle Y , \overline{\mean} \rangle Z,
\quad 
A_{\overline{Y}} \Zba
= \overline{A_Y Z}, \quad Z  \in T^{1,0}M,\ Y\in N^{1,0}M.
\end{equation}
Moreover, $A_{Y}$ maps $T^{1,0}M \oplus T^{0,1}M$ into $T^{1,0}M$ if $Y$ is of type (1,0).
The \textit{normal connection} $D$ on $N^{1,0}M \oplus N^{0,1}M$ is then defined by
\begin{align}
D_ZY
=
\widetilde{\nabla}_Z Y + A_{Y} Z.
\end{align}
Then $D_ZY \in T^{1,0}\mathcal{X}$. Moreover, for $\Wba \in T^{0,1}M$, we have 
\begin{align}
	\langle D_ZY , \Wba \rangle
	& =
	\langle \widetilde{\nabla}_ZY , \Wba \rangle + \langle A_YZ , \Wba \rangle \notag \\
	& =
	Z \cdot \langle Y, \Wba \rangle - \langle Y , \widetilde{\nabla}_Z \Wba \rangle + \langle \II(Z, \Wba) , Y \rangle \notag \\
	& =
	0.
\end{align}
Hence, $D_ZY \in N^{1,0}M$ whenever $Y\in N^{1,0}M$ and $Z \in T^{1,0}M \oplus T^{0,1}M$. By usual arguments, we can show that $D$ is a linear connection on $N^{1,0}M \oplus N^{0,1}M$ which respects the splitting into $(1,0)$ and $(0,1)$ parts of complex vector fields. Furthermore, we can also verify that $D$ is metric, i.e., 
\begin{equation}
	X \cdot \langle Y, \Zba \rangle 
	=
	\langle D_XY , \Zba \rangle + \langle Y , D_X\Zba \rangle, \quad Y \in N^{1,0}M,\ \Zba \in N^{0,1}M.
\end{equation}
Details are left to the readers.
\begin{prop}\label{prop:constantmean}
	It holds that
	\begin{equation}\label{e:dzovh}
	D_{Z} \overline{H} = 0,
	\end{equation}
for all (1,0)-vectors in $T^{1,0}M$. Consequently, if $D_ZH = 0$, then $|H|$ is a constant.
\end{prop}
\begin{proof}
It follows from \cref{e:tauz} that $\widetilde{\nabla}_Z \overline{H} = -i\tau Z$ is tangent to $M$. Therefore, $D_{Z} \overline{H} = (\widetilde{\nabla}_Z \overline{H})^{\perp} = 0$. Moreover, if $D_ZH = 0$ for all (1,0) tangent vector $Z$, then
\begin{align}
Z \cdot |H|^2
=
\langle D_Z H , \overline{H} \rangle 
+
\langle H, D_Z \overline{H} \rangle 
=
0,
\end{align}
and thus $|H|^2$ is an real-valued anti CR function, hence constant.
\end{proof}

Using the normal connection $D$, we can rewrite \cref{e:2curv} as follows
\begin{align}\label{e:normalcur}
\widetilde{R}(X, \overline{Y}) Z
&=
{R}(X, \overline{Y}) Z
+
A_{\II(X,Z)} \overline{Y} - A_{\II(\overline{Y}, Z)} X \notag \\
& \qquad +
(D_{X} \II) (\overline{Y}, Z) - (D_{\overline{Y}} \II) (X, Z) + \II(\mathbb{T}_{\nabla}(X, \overline{Y}), Z).
\end{align}
Here 
\begin{equation}
	(D_X\II)(Y,Z) = D_X(\II(Y,Z)) - \II(\nabla_X Y , Z) - \II(Y, \nabla_XZ),
\end{equation}
for $X,Y,Z$ are sections of $T^{1,0}M \oplus T^{0,1}M$.
\begin{prop}[Codazzi-Mainardi equation]\label{prop:codazzi-mainardi}
	If $\iota \colon (M, \theta) \hookrightarrow (\mathcal{X},\omega)$ is a pseudohermitian submanifold and $\omega$ is Kähler, then the normal component of the curvature is
\begin{align}\label{e:normal2}
	(\widetilde{R}(X, \overline{Y}) Z)^{\perp}
	=
	- (D_{\overline{Y}} \II) (X, Z) + \langle \overline{Y}, Z \rangle D_X H + 
	\langle X, \overline{Y} \rangle D_Z H.
\end{align}
\end{prop}
\begin{proof}
 	From the fact that $\mathbb{T}_{\nabla}$ is pure and \cref{e:iitz1}, we have
	\begin{align}
		\II(\mathbb{T}_{\nabla}(X, \overline{Y}), Z)
		=
		i\langle X, \overline{Y} \rangle \II(T, Z)
		=
		\langle X, \overline{Y} \rangle \widetilde{\nabla}_Z H.
	\end{align}
	Taking the normal components, 
	\begin{align}
		\II(\mathbb{T}_{\nabla}(X, \overline{Y}), Z)^{\perp}
		=
		\langle X, \overline{Y} \rangle D_Z H.
	\end{align}
	On the other hand, since $\II(\overline{Y}, Z) = \langle \overline{Y}, Z \rangle \mean$,
	\begin{align}
	(D_{X} \II) (\overline{Y}, Z)
	&=
	D_X (\langle \overline{Y} , Z \rangle H ) - \II(\nabla_X \overline{Y} , Z) - \II (\overline{Y} , \nabla_X Z)	\notag \\
	&=
	\langle \overline{Y} , Z \rangle D_X H.
	\end{align}
	Here we used the fact that the connection $D$ is metric. 
	
	Then \cref{e:normal2} follows from taking the normal component of \cref{e:normalcur}.
\end{proof}
Take an orthogonal coframe $\{\theta^{\alpha}\}$ of $(T^{1,0}M)^{\ast}$ and its conjugate $\{\theta^{\bba}: = \overline{\theta^{\beta}}\}$. Thus, $\{\theta^{\alpha}, \theta^{\bba}, \theta\}$ is an orthonormal coframe of the complexified cotangent bundle $(\mathbb{C}TM)^{\ast}$. The dual frame will be denoted by $\{Z_{\alpha}, Z_{\bba}, Z_0 = T\}$. In this frame, the \textit{pseudohermitian curvature tensor} has components
\begin{equation}
R_{\alpha\bba\gamma\sba}
=
\left\langle \nabla_{\alpha}\nabla_{\bba} Z_{\gamma} - \nabla_{\bba}\nabla_{\alpha} Z_{\gamma} - \nabla_{[Z_{\alpha},Z_{\bba}]} Z_{\gamma} , Z_{\sba} \right\rangle.
\end{equation}
The Ricci tensor is $R_{\alpha\bba} = R_{\alpha}^{\sba}{}_{\bba}{}_{\sba}$ and the Ricci $(1,1)$-form, denoted by $\Ric$, is a $(1,1)$-form on $T^{0,1}M \oplus T^{1,0}M$ which agrees with $iR_{\alpha\bba}\theta^{\alpha}\wedge\theta^{\bar{\beta}}$ when restricted to $T^{1,0}M\oplus T^{0,1}M$. Unlike its Kähler counterpart, the pseudohermitian Ricci form does not necessary extend to a closed $(1,1)$-form. The components of the pseudohermitian torsion $\tau$ are denoted by~$A_{\alpha\beta}$. Precisely,
\begin{equation}
	A_{\alpha\beta} : = \langle \tau Z_{\alpha} , Z_{\beta} \rangle.
\end{equation}
It is well-known that $A_{\alpha\beta}$ is symmetric, i.e., $A_{\alpha\beta} = A_{\beta\alpha}$ (see, e.g., \cite{webster1978pseudo,lee1988pseudo}).

We denote by $\omega_{\alpha\gamma}^a$ the components of the ``holomorphic'' part of the second fundamental form $\II$, i.e.,
\begin{equation}
	\II(Z_{\alpha},Z_{\gamma}) 
	=
	\omega_{\alpha\gamma}^a Z_a,
\end{equation}
where we sum over the lowercase Latin index which runs from $n+1$ to $N$. Here $\{Z_1,\dots,Z_n\}$ is an orthonormal frame for $T^{1,0}M$ and $\{Z_1, \dots ,Z_{N}\}$ is an orthonormal frame for $T^{1,0}\mathcal{X}$. Then Gauß equation takes the following form:
\begin{equation}
	\widetilde{R}_{\alpha\bba\gamma\sba}
	=
	{R}_{\alpha\bba\gamma\sba}
	+
	\omega_{\alpha\gamma}^a \omega_{\bba\sba}^{\bar{b}} h_{a\bar{b}}
	-
	|\mean |^2\left(h_{\alpha\bba} h_{\gamma\sba} + h_{\alpha\sba} h _{\gamma\bba}\right).
\end{equation} 
Moreover,
\begin{equation}
	A_{\alpha\beta} = -i \omega^{a}_{\alpha\beta} \mean ^{\bar{b}} h_{a\bar{b}}.
\end{equation}
We obtain the following CR analogues of the well-known inequalities for isometric Riemannian immersions into the real euclidean space.
\begin{cor}\label{cor:riclower} Let $(M,\theta) \hookrightarrow (\mathbb{C}^N, \omega: = i\partial\bar{\partial} \|Z\|^2)$ be a semi-isometric CR immersion. Let $\Ric$ and $R$ be the Ricci form and the Webster scalar curvature, respectively. Then
\begin{equation}
	\Ric \leq (n+1) |\mean |^2(\iota^{\ast}\omega)|_{H(M)}\ \text{and}\
	R \leq n(n+1) |\mean |^2.
\end{equation}
The equality in each of them occurs iff $M$ is totally umbilical.
\end{cor}

\subsection{CR immersions into the sphere}
Let $\iota \colon \mathbb{S}^{2n+1} \subset \mathbb{C}^{n+1}$ be the standard embedding of the standard CR sphere into the complex space. Let $\II$ be the corresponding second fundamental form and $\mean_{\mathbb{S}^{2n+1}}$ the (1,0)-mean curvature field. Then
\begin{equation}\label{e:tc}
	\II (Z,W) = 0.
\end{equation}
Indeed, take the standard coordinates $(z^1, \dots , z^n , z^{n+1} = w)$ on $\mathbb{C}^{n+1}$ and $\rho : = \|Z\|^2 - 1$. Then $\rho$ is a defining function for the sphere and the 
standard pseudohermitian structure is $\iota^{\ast}(i\bar{\partial} \rho)$.
Clearly, 
\begin{equation}
	d\theta = \iota^{\ast}(i\partial\bar{\partial}\rho) = \iota^{\ast} \omega.
\end{equation}
Hence, the inclusion $\iota$ is semi-isometric.

A basis for $(1,0)$-vectors on $\mathbb{S}^{2n+1}$ is the restrictions onto $\mathbb{S}^{2n+1}$ of 
\begin{equation}
	Z_{\alpha} : = \partial_\alpha - (\zba^{\alpha}/\wba) \partial_w, \quad \alpha = 1,2,\dots , n,
\end{equation}
at points where $\rho_w \ne 0$.

If $\widetilde{\nabla}$ is the Chern connection of $\mathbb{C}^{n+1}$, then 
\begin{equation}\label{e:chern}
	\widetilde{\nabla}_{Z_{\alpha}} Z_{\beta} = \widetilde{\nabla}_{Z_{\alpha}} \partial_\beta - Z_{\alpha} \left(\frac{\zba^{\beta}}{\wba}\right) \partial_w - \left(\frac{\zba^{\beta}}{\wba}\right) \widetilde{\nabla}_{Z_{\alpha}} \partial_w
	=
	0.
\end{equation} 
On the other hand, if $\nabla$ is the Tanaka-Webster connection on $\mathbb{S}^{2n+1}$ and $\omega_{\beta}{}^{\gamma}$'s are the connection forms associated to the chosen frame, then by \cite{li--luk} (eq. \cref{e:cf} below)
\begin{equation}
	\omega_{\beta}{}^{\gamma}(Z_{\alpha}) 
	=
	h^{\gamma\bar{\mu}} Z_{\alpha} h_{\beta\bar{\mu}} - \xi_{\beta} \delta_{\alpha}^{\gamma},
\end{equation} 
where $h_{\beta\bar{\mu}}$ is the Levi-matrix:
\begin{equation}
	h_{\beta\bar{\mu}} = \delta_{\beta\mu} + \frac{\zba^{\beta} z^{\mu}}{|w|^2},
	\quad
	h^{\gamma\bar{\mu}}
	=
	\delta_{\gamma\mu} - z^{\gamma}\zba^{\mu}.
\end{equation}
Consequently
\begin{equation}
	\omega_{\beta}{}^{\gamma}(Z_{\alpha}) 
	=
	h^{\gamma\bar{\mu}} Z_{\alpha} h_{\beta\bar{\mu}}
	- \xi_{\beta} \delta_{\alpha}^{\gamma}
	=
	\frac{1}{|w|^2} \zba^{\beta} \delta_{\alpha}^{\gamma}
	-
	\left(\delta_{\beta\sigma} + \frac{\zba^{\beta} z^{\sigma}}{|w|^2}\right) \zba^{\sigma} \delta_{\alpha}^{\gamma}
	=
	0.
\end{equation}
This and \cref{e:chern} imply that $\II(Z_{\alpha}, Z_{\beta}) = 0$, as desired.
\begin{prop}\label{prop:2sff}
	Let $M$ be a strictly pseudoconvex CR manifold and $\phi \colon M \to \mathbb{S}^{2N-1}$ a CR immersion. Let $\iota \colon \mathbb{S}^{2N - 1} \to \mathbb{C}^{N}$ be the standard inclusion. Put
	\begin{equation}
		F: = \iota \circ \phi ,
		\quad
		\theta = F^{\ast} \Theta,
		\quad
		\omega : = i\partial\bar{\partial} \|Z\|^2.
	\end{equation} 
	Then $F \colon (M,\theta) \to (\mathbb{C}^{N}, \omega)$ is a semi-isometric CR immersion. Moreover, if $\II_{M}^{C\!R}$ is the CR second fundamental form of $\phi$ for any admissible pair $(\theta,\hat{\theta})$, then 
	\begin{equation}
		\II_{M}^{C\!R}(Z,W) = \II_{M}^F(Z,W),
	\end{equation}
	for every pair $Z,W$ in $T^{1,0}M$.
\end{prop}
\begin{proof} Suppose that $(\theta', \hat{\theta})$ is an admissible pair, in the sense of \cite{ebenfelt2004rigidity}, of pseudohermitian structures for the CR immersion $\phi \colon M \to \mathbb{S}^{2N - 1}$. This means that $\theta' = \phi^{\ast}\hat{\theta}$ and the Reeb vector field of $\hat{\theta}$ is tangent to $\phi(M)$ along the image. The CR second fundamental form is given by
	\begin{equation}
	\II_M^{C\!R} (Z,W)
	: = \nabla^{\hat{\theta}}_{Z} W - \nabla^{\theta}_{Z} W ,
	\end{equation}
	for all $(1,0)$-vector fields $Z, W$ tangent to $M$ and extended smoothly to $(1,0)$-vector fields of $\mathbb{S}^{2N-1}$. Suppose that $\hat{\theta} = e^u \Theta$. Then $\theta = e^{-u}\theta'$. By Lee's formula for pseudo-conformal change of contact forms (see, e.g., \cite{dragomir--tomassini}),
	\begin{equation}
	\nabla^{\hat{\theta}}_{Z} W = \nabla^{\Theta}_{Z} W + Z(u) W + W(u) Z,
	\end{equation}
	and a similar identity holds on $M$. Using \cref{e:tc}, we obtain
	\begin{align}
	\II_M^{C\!R}(Z,W)
	& = 
	\nabla^{\Theta}_ZW - \nabla^{\theta}_ZW \notag \\
	& =
	\widetilde{\nabla}_ZW - \nabla^{\theta_0}_ZW \notag \\
	& = \II_{M}(Z,W).
	\end{align}
	Here we identify $Z$ and $W$ with their push-forwards via $F$ and $\phi$. In particular, the CR second fundamental form $\II_M^{C\!R}(Z,W)$ agrees with the holomorphic part of $\II$ and can be computed using any pair $\hat{\theta}$ and $\theta': = F^{\ast} \hat{\theta}$.
\end{proof}
\begin{example}\label{ex:elip}
 Let $(E(A), \theta)$ be a real ellipsoid defined by $\rho: = \|z\|^2 - \Re (Az \cdot (Az)^t) - 1 = 0$ and $\theta = \iota^{\ast} (i\bar{\partial} \rho)$, where $A = (A_1,A_2,\dots ,A_{n+1})$ are real numbers. Then the inclusion $\iota$ is a semi-isometric embedding of $(E(A), \theta)$ into $\mathbb{C}^{n+1}$ with euclidean metric. Observe that $E(A)$ also admits a CR immersion into $\mathbb{S}^{2n+3}$, but not into $\mathbb{S}^{2n+1}$ in general, and so $\iota$ does \textit{not} realize an immersion in a sphere.
\end{example}

\subsection{The transverse curvature of a level set of a Kähler potential}\label{sec:transverse}
\begin{prop}\label{prop:kahlerpotential}
Let $M \subset \mathbb{C}^{n+1}$ be a strictly pseudoconvex real hypersurface defined by $\rho = 0$ with $d\rho \ne 0$ and let $\theta = \iota^{\ast} (i\bar{\partial} \rho)$. Assume that $F \colon (M,\theta) \to (\mathcal{X},\omega)$ is a semi-isometric CR immersion and $\varphi$ is a local Kähler potential for $\omega$ on a neighborhood of $F(p) \in \mathcal{X}$ ($p\in M$). Then there is a CR function $G$ in a neighborhood of $p$ such that $\varphi \circ F = \Re G$.
\end{prop}
\begin{proof} We assume that $U$ is an open set in $\mathbb{C}^{n+1}$, $F$ sends $U\cap M$ into an open coordinate patch $V \subset \mathcal{X}$, and $\varphi$ is defined in $V$ such that $i\partial\bar{\partial}\varphi = \omega$. Since $M$ is strictly pseudoconvex, we can apply H. Lewy extension theorem to each component of $Z\circ F$ ($Z$ is a holomorphic coordinate in $V$) to deduce that $F$ extends to a holomorphic map $\tilde{F}$ in an one-sided neighborhood $U^+\subset U$ of $p$. (Note that all components of $Z \circ F$ extend to the same side, the pseudoconvex side). We can also assume that $\tilde{F}(U) \subset V$. Observe that $F=\tilde{F}\circ\iota$ and hence $F^{\ast} = \iota^{\ast} \circ \tilde{F}^{\ast}$, by the smoothness of the extension. Since $\theta=\iota^{\ast}(i\bar\partial\rho)$, the semi-isometry assumption gives
\[
    \iota^{\ast}\tilde{F}^{\ast}\partial\bar\partial\varphi=\iota^{\ast}\partial\bar\partial\rho.
\]
As $\tilde{F}$ is holomorphic on $U^+$, the pull-back $\tilde{F}^{\ast}$ commutes with $\partial$ and $\bar\partial$, and by continuity, the same holds on $M\cap U$.
Hence $\iota^{\ast} (i\partial\bar{\partial}(\varphi \circ \tilde{F} - \rho)) =0$ and thus $\varphi \circ F$ satisfies conditions (2) and (3) in Bedford and Federbush \cite{bedford1974pluriharmonic} with $\alpha = 1$. Hence, by Theorem 1 of \cite{bedford1974pluriharmonic}, locally there exists a CR function $G$ such that $\varphi \circ F = \Re G$.

Alternatively, one can verify, using \cref{e:cf} below, that $\varphi \circ F$ satisfies the condition characterizing CR-pluriharmonic functions in \cite{lee1988pseudo} and hence the existence of such $G$ follows. The proof is complete.
\end{proof}
In view of \cref{prop:kahlerpotential}, if $(M,\theta)$ is semi-isometric CR immersed into a complex euclidean space $\mathbb{C}^N$, then $M$ is locally CR embeddable into a sphere of $\mathbb{C}^{N+1}$. Indeed,  $\|Z\|^2$ is a K\"ahler potential of the euclidean metric in $\mathbb{C}^N$. If $F = (F_1, F_2,\dots , F_N)$ is such an immersion into $\mathbb{C}^N$ and $G$ is a local CR function on $M$ such that $\|F\|^2 = \Re G$ on $M$. Then the map $(F_1, \dots, F_N, G)$ is a CR map sending $M$ into the Heisenberg hypersurface defined by $\Re Z_{N+1} = \sum_{A=1}^N |Z_A|^2$ in $\mathbb{C}^{N+1}$. Since CR manifolds are not generally CR embeddable into a sphere of any dimension, even locally \cite{faran1988nonimbeddability}, the CR analogue of the Cartan-Janet theorem for semi-isometric CR immersions does not hold.

We consider a strictly pseudoconvex
real hypersurface $M$ defined by $\rho = 0$ with $d\rho \ne 0$ and $\rho$ is a Kähler potential for a metric $\omega$ on $U$. If $\theta : = (i/2)(\bar{\partial}\rho - \partial \rho)$, then $\iota$ is CR semi-isometric. As pointed out in \cite{lee--melrose}, there is an associated $(1,0)$-field $\xi$ such that 
\begin{equation}\label{e:tvdef}
\xi \, \rfloor \, \partial \rho = 1, \quad \xi \, \rfloor \, \partial\bar{\partial} \rho = 0 \mod \bar{\partial} \rho.
\end{equation}
The function $r: = \rho_{j\kba} \xi^j \xi^{\kba}|_M$ is called the transverse curvature \cite{graham1988smooth}. Take a local holomorphic coordinate $z_1,z_2,\dots , z_n, z_{n+1} = w$ such that $\rho_w \ne 0$ and define
\begin{equation}\label{e:cframe}
\theta^{k}: = dz^k - \xi^k \partial\rho, \quad k = 1,2,\dots , n+1.
\end{equation}
Then $\{\theta^{\alpha}$, $\alpha = 1,2,\dots , n\}$ is an admissible coframe for $(M,\theta)$. The corresponding Levi matrix $h_{\alpha\bba}$ is given by
\begin{equation}\label{e:levimatrix0}
	h_{\alpha\bba} =
	-i d\theta(Z_{\alpha}, Z_{\bba})
	=
	\rho_{\alpha \bba}-\rho_\alpha \partial_{\bba}\log \rho_{w}-\rho_{\bba}\partial_{\alpha}\log \rho_{\wba}+\rho_{w\wba}
	\frac{\rho_\alpha \rho_{\bba}}{|\rho_{w}|^2}.
\end{equation}
Moreover, the Tanaka-Webster connection forms $\omega_{\beta}{}^{\alpha}$'s are given by \cite{li--luk}
\begin{align}\label{e:cf}
	\omega_{\beta}{}^{\alpha}
	=
	\left(h^{\alpha\bar{\mu}}Z_{\gamma} h_{\beta\bar{\mu}} -\xi_{\beta}\delta_{\gamma}^{\alpha}\right) \theta^{\gamma}
	+ \xi^{\alpha} h_{\beta\bar{\gamma}} \theta^{\bar{\gamma}}
	- i Z_{\beta} \xi^{\alpha}\theta.
\end{align}
From \cref{e:cf}, we can calculate the Ricci form via the formula $\Ric = id\omega_\alpha{}^{\alpha} \mod \theta$.
Indeed, Li and Luk \cite{li--luk} derived the following useful formula. 
\begin{prop}[Li--Luk \cite{li--luk}] Let $\Ric$ be the Ricci (1,1)-form restricted to $H(M)$. Then
\begin{equation}
\Ric
=
(n+1)r(\iota^{\ast}\omega)|_{H(M)}
-
\iota^{\ast} (i\partial\bar{\partial} \log J(\rho))|_{H(M)}.
\end{equation} 
Here $J(\rho)$ is the (Levi-) Fefferman determinant defined in \cref{e:fm}.
\end{prop}
We calculate the $(1,0)$-mean curvature vector explicitly as follows.
\begin{prop}\label{prop:transmean}
	Let $M\subset \mathbb{C}^{n+1}$ be defined by $\rho=0$, $d\rho \ne 0$, where $\rho$ is a strictly plurisubharmonic defining function on an open set $U$ containing $M$. Assume that $\theta = \iota^{\ast}(i\bar{\partial}\rho)$ and $\omega = i\partial\bar{\partial}\rho$. Then $\iota \colon (M,\theta) \to (U , \omega)$ is a semi-isometric CR immersion. Moreover, the second fundamental form satisfies
	\begin{equation}\label{e:sfftrans}
		\II(Z_{\aba}, Z_{\beta})
		= 
		- h_{\beta\aba}\, \xi.
	\end{equation}
	In particular, the squared mean curvature $|\mean|^2$ coincides with the transverse curvature of $\rho$:
	\begin{equation}
	r(\rho) = |\xi|_{\omega}^2
	=
	|\mean|_{\omega}
	^2.
	\end{equation}
\end{prop}
\begin{proof} It is immediate that $\iota$ is a CR semi-isometric immersion as $d\theta = \iota^{\ast} \omega$. 
To prove \cref{e:sfftrans}, we shall calculate the second fundamental form explicitly by using \cref{e:cf}. Let
$(z^1, z^2, \dots , z^n , w = z^{n+1})$ be a local coordinate system near a point $p\in M$.
We can suppose that $\rho_w \ne 0$ near $p$. Put 
\begin{equation}
	Z_{\gamma} = \partial_{\gamma} - (\rho_\gamma/\rho_w)\partial_w.
\end{equation}
Then $Z_{\alpha}$, $\alpha = 1,2, \dots , n$, form a frame of $T^{1,0}M$ near a point $p$ with $\rho_w \ne 0$. From this, one can easily compute,
\begin{align}
	Z_{\alpha}\left(\frac{\rho_{\bba}}{\rho_{\wba}}\right)
	=
	\left(\partial_{\alpha} - \frac{\rho_{\alpha}}{\rho_w}\partial_w\right) \left(\frac{\rho_{\bba}}{\rho_{\wba}}\right) = 
	\frac{h_{\alpha\bba}}{\rho_{\wba }}.
\end{align}
Therefore, if $\widetilde{\nabla}$ is the Chern connection of $\omega$, then
\begin{align}
	\widetilde{\nabla}_{Z_{\alpha}} Z_{\bba }
	& =
	\widetilde{\nabla}_{Z_{\alpha}} \partial_{\bba} - Z_{\alpha}\left(\frac{\rho_{\bba}}{\rho_{\wba}}\right) \partial_{\wba} - \left(\frac{\rho_{\bba}}{\rho_{\wba}}\right) \widetilde{\nabla} _{Z_{\alpha}} \partial_{\wba}\notag \\
	& =
	-(1/\rho_{\wba }) h_{\alpha\bba } \partial_{\wba }.
\end{align}
	On the other hand, by \cref{e:cf}, we have
	\begin{align}
	\nabla_{Z_{\alpha}}Z_{\bba }\notag 
	& = 
	\xi^{\bar{\gamma}} h_{\alpha\bba } Z_{\bar{\gamma}} \notag \\
	& = 
	|\partial \rho|^{-2}_{\rho} h_{\alpha\bba } \rho^{\bar{\gamma}} \left(\partial_{\bar{\gamma}} -(\rho_{\bar{\gamma}} / \rho_{\wba })\partial_{\wba } \right) \notag \\
	& = h_{\alpha\bba } \left[r \rho^{\kba } \partial_{\kba } - \frac{1}{\rho_{\wba }} \partial_{\wba }\right].
	\end{align}	
	Therefore,
	\begin{align}\label{e:sffum}
	\II(Z_{\alpha}, Z_{\bba })
	& =
	\widetilde{\nabla}_{Z_{\alpha}} Z_{\bba }
	-
	\nabla_{Z_{\alpha}} Z_{\bba } \notag \\
	& =
	- r h_{\alpha\bba } \rho^{\kba } \partial_{\kba } \notag \\
	& =
	- h_{\alpha\bba } \overline{\xi}.
	\end{align}
	This proves \cref{e:sfftrans}.	Taking the trace with respect to the Levi-form, we obtain
	\begin{equation}
	\overline{\mean} = - \overline{\xi}.
	\end{equation}
	In particular,
	\begin{equation*}
	|\mean|^2_{\omega}
	=
	|\bar{\xi}|^2_{i\partial\bar{\partial} \rho}
	=
	r(\rho). \qedhere 
	\end{equation*}
\end{proof}
\begin{cor}	
	Let $F \colon U \to (\mathcal{X},\omega)$ be a holomorphic immersion and $\varphi$ a Kähler potential of $\omega$. Let $\rho: = \varphi \circ F$ and suppose that $\{\rho = 0\}$ defines a strictly pseudoconvex real-hypersurface $M\subset U$ with $d\rho \ne 0$ on $M$. Put $\theta: = i\bar{\partial} \rho|_M$ then $F$ is a semi-isometric immersion from $(M,\theta)$ into $(\mathcal{X},\omega)$. Moreover, 
	\begin{equation}
	|\mean_{F(M)}|^2\circ F = r(\rho),
	\end{equation}
	where $r(\rho)$ is the transverse curvature.
\end{cor}

\begin{proof}
From \cref{e:meanreeb}, we have $\mean_{F(M)} = -F_{\ast} \xi$, where $\xi$ is the transverse vector field associated to the definining function~$\rho$. In local coordinates $Z^A$ of $\mathcal{X}$, we write $F^A = Z^A \circ F$. Then 
\begin{equation}
	|\mean_{F(M)}|^2 \circ F
	=
	\varphi_{A\bar{B}} F^A_j\xi^j F^{\bar{B}}_{\bar{k}} \xi^{\bar{k}} = \rho_{j\kba}\xi^j\xi^{\kba} = r(\rho).
\end{equation}
Here the repeated uppercase indices are summed from $1$ to $\dim_{\mathbb{C}} \mathcal{X}$ while the lowercase indices are summed from $1$ to $\dim_{C\!R}M$.
\end{proof}

\section{Chern-Moser CR umbilical points and umbilical points of immersions}\label{sec:um}
In this section, we use the Gauß equation to determine the CR umbilical points on strictly pseudoconvex CR manifolds of dimension at least 5 which are semi-isometrically immersed in a complex euclidean space and prove \cref{thm:umbilichypersurface}.
\begin{defn}[Chern-Moser CR umbilical points \cite{chern1974real}]
	Let $M$ be a Levi-nondegenerate CR manifold of hypersurface type. A point $p\in M$ is called a CR umbilical point if the Chern-Moser curvature tensor vanishes at $p$.
\end{defn}
It is well-known that if $M$ is CR umbilical in a neighborhood of $p$, then $M$ is locally spherical at $p$ \cite{chern1974real}.

We denote $\II^{\circ}$ the traceless component of $\II$. Precisely, $\II^{\circ}(Z,W) = \II(Z,W)$, $\II^{\circ}(\Zba, \Wba) = \II(\Zba, \Wba)$, and $\II^{\circ}(Z,\Wba) = \II^{\circ}(\Wba, Z) = 0$, for all $Z, W \in T^{1,0}M$ and $\Zba, \Wba \in T^{0,1}M$.
\begin{defn}[Umbilical points of an immersion]
	Let $\iota \colon (M,\theta) \hookrightarrow (\mathcal{X},\omega)$ be a strictly pseudoconvex pseudohermitian CR submanifold, $d\theta = \iota^{\ast} \omega$. We say that $p\in M$ is \emph{pseudohermitian umbilical} at $p$ if $\II^{\circ}(p) = 0$.
\end{defn}
The following is a simple extension of Lemma 5.2 in \cite{ebenfelt2004rigidity}.
\begin{prop}\label{prop:2um}
	Let $\iota \colon (M^{2n+1},\theta) \hookrightarrow (\mathbb{C}^N, \omega)$, $\omega:= i\partial\bar{\partial}\|Z\|^2$, be a strictly pseudoconvex pseudohermitian submanifold and $p\in M$. Assume that $\dim M \geq 5$.
	\begin{enumerate}[(i)]
		\item If $\II^{\circ}(p) = 0$, then $p$ is a CR umbilical point in the sense of Chern and Moser.
		\item If $M$ is a CR umbilical at $p$ and $N\leq 2n$, then $\II^{\circ}(p) = 0$.
	\end{enumerate}
\end{prop}
\begin{proof} If $\II^{\circ}(p) = 0$, then the Gauß equation at $p$ reduces to
	\begin{equation}
	R_{\alpha\bba\gamma\sba}\bigl|_p
	=
	|\mean|^2(h_{\alpha\bba}h_{\gamma\sba} + h_{\alpha\sba} h_{\gamma\bba})\bigl|_p.
	\end{equation}
	This implies that the traceless component of $R_{\alpha\bba\gamma\sba}$ vanishes at $p$ and hence (i) follows.
	
	Suppose that $p$ is CR umbilical and $N\leq 2n$. Then taking the traceless component of both sides of the Gauß equation \cref{e:gauss}, we obtain
	\begin{equation}
	\tf \omega_{\alpha\gamma}^a \omega_{\bba\sba}^{\bar{b}} h_{a\bar{b}}\bigl|_p = 0.
	\end{equation}
	Since $N\leq 2n$, we can argue as in Lemma~5.3 of \cite{ebenfelt2004rigidity}, using Huang's lemma, to deduce that $\omega_{\alpha\gamma}^a = 0$ at $p$. This proves (ii).
\end{proof}
\begin{cor}\label{cor:2um}
	Let $\iota \colon (M,\theta) \hookrightarrow (\mathbb{C}^N, \omega)$, $\omega:= i\partial\bar{\partial}\|Z\|^2$, be a strictly pseudoconvex pseudohermitian submanifold and $p\in M$. Suppose that $N\leq 2n$, then the following are equivalent:
	\begin{enumerate}[(i)]
		\item $p$ is a CR umbilical point of $M$,
		\item $\Ric|_p = (n+1)|\mean|^2 \iota ^{\ast}\omega|_p$,
		\item $R|_p = n(n+1)|\mean|_p^2$,
	\end{enumerate}
	Moreover, each implies that $A_{\alpha\gamma}|_p = 0$. Here $A_{\alpha\gamma}$ are the components of the pseudohermitian torsion.
\end{cor}
\begin{proof}
	Assume that (i) holds, then $\II^{\circ}(p) = 0$ by \cref{prop:2um}. Tracing the Gauß equation at $p$ implies that
	\begin{equation}
	\Ric = (n+1) |\mean|^2 \iota^{\ast}\omega.
	\end{equation} 
	This is (ii).
	
	Clearly (ii) implies (iii) by taking the trace.
	
	Assume that (iii) holds. The Gauß equation at $p$ implies that 
	\begin{equation}
		R = n(n+1) |\mean |^2 - |\II^{\circ}|^2.
	\end{equation} 
	Hence, (iii) implies that $\II^{\circ} = 0$ and hence $p$ is CR umbilical by \cref{prop:2um}.
	
	The last conclusion follows from \cref{e:gausstorsion}.
\end{proof}
\begin{proof}[Proof of \cref{thm:umbilichypersurface}]
By assumption, 
\begin{equation}
	\rho = \sum_{d = 1}^{N} |F^d|^2 + \psi,
\end{equation}
where $\psi$ is real-valued, $i\partial\bar{\partial} \psi =0$, and $F^d$'s are holomorphic in a neighborhood of $M$. Let $\theta = i\bar{\partial}\rho$, then $d\theta = \iota^{\ast} (i\partial\bar{\partial}\rho) = F^{\ast} (i\partial\bar{\partial}\|Z\|^2)$, and hence $F: = (F^1, \dots , F^N)$ is a semi-isometric CR immersion from $(M,\theta)$ into $(\mathbb{C}^N, i\partial\bar{\partial}\|Z\|^2)$. We then identify $M$ with its image $F(M) \subset \mathbb{C}^N$. By Li-Luk's formula,
\begin{equation}
	\Ric
	=(n+1) r(\rho) (\iota^{\ast} \omega)|_{H(M)}
	- \iota^{\ast}(i\partial\bar{\partial} \log J(\rho))|_{H(M)}.
\end{equation}
On the other hand, $|\mean|^2 = r(\rho)$ by \cref{prop:transmean}. Taking the trace of Gauß equation \cref{e:gauss}, we obtain
	\begin{align}\label{e:j}
	\iota^{\ast}(i\partial\bar{\partial} \log J(\rho))|_{H(M)}
	& =
	(n+1) |H|^2 (\iota^{\ast} \omega)|_{H(M)} - \Ric \notag \\
	& =
	h^{\alpha\bba}\omega_{\alpha\gamma}^a \omega_{\bba\sba}^{\bar{b}}h_{a\bar{b}} \, \theta^{\gamma} \wedge \theta^{\sba}|_{H(M)}.
	\end{align}
	The last expression is manifestly nonnegative as a $(1,1)$-form and hence \cref{e:umbilic} follows. The equality occurs if and only if $\omega_{\alpha\gamma}^a = 0$ if and only if the trace with respect to $h^{\gamma\sba}$ of each side of \cref{e:j} vanishes.
	
	When $N \leq 2n$, the equality in \cref{e:umbilic} occurs at $p$ if and only if $p$ is a CR umbilical point, by \cref{cor:2um}. The proof is complete.
\end{proof}
\begin{proof}[Proof of \cref{cor:webster}]
	Let $\rho := \|z_j\|^2 - \Re (z^t A z ) -1$, $A = \mathrm{diag}(A_1,A_2,\dots, A_n)$, be a defining function for $E$ which satisfies the condition in \cref{thm:umbilichypersurface}. Explicitly,
	\begin{equation}
	(\log (J(\rho))_{j\kba}
	=
	|\partial \rho|^2 A_jA_k \delta_{jk} - A_jA_k \rho_{\jba}\rho_{k} \geq 0
	\end{equation}
	(as an inequality of Hermitian matrices.) Moreover, if there are at least two nonzero in $A_k$'s, then the complex Hessian of $\log J(\rho)$ has at least two positive eigenvalues (at every points on $E(A)$) and hence the proof follows.
\end{proof}
\begin{rem}\label{rem:elip}
If $A$ has exactly one nonzero element, say $A_{1} \ne 0$, then $p$ is a CR umbilical point iff $|\partial \rho|^2 - |\rho_1|^2 = 0$ at $p$. This is the case iff $z_2 =\cdots = z_{n+1} =0$ and $z_1$ verifies $|z_1|^2 + \Re (A_1 z_1^2) =1$. The CR umbilical locus is an ellipse in the $z_1$-coordinate plane.
\end{rem}
\section{A Beltrami type formula for $\Box_b$ and a Takahashi type theorem}\label{sec:bel}
Explicit formulas for the Kohn Laplacian that are analogous to the well-known Beltrami formula for the Laplacian were derived by Li, Lin, and the author in \cite{li--son,li--lin--son}. In the notations of this paper, they can be reformulated as follows.
\begin{prop}[cf. \cite{li--son,li--lin--son}]\label{prop:beltrami}
	Let $(\mathcal{X}, \omega)$ be a Kähler manifold, $\iota \colon M \to \mathcal{X}$ a semi-isometric CR immersion, and $\mean$ the corresponding $(1,0)$-mean curvature field. If $f$ is the restriction of a (possibly of complex-valued) pluriharmonic function $\widetilde{f}$ in a neighborhood of~$M$ in $\mathcal{X}$, then
	\begin{equation}\label{e:klpluriharmonic}
	\Box_b f 
	=
	-n\, \overline{\mean} \widetilde{f}.
	\end{equation}
	In particular, if $\{Z^A\colon A =1,2,\dots , N\}$ is a local holomorphic coordinate system in a neighborhood of $M$, then
	\begin{equation}\label{e:belform}
	\overline{\mean} = -\frac{1}{n} \sum_{A=1}^{N} \Box_b \left(\overline{Z}^A|_M \right) \partial_{\bar{A}}.
	\end{equation}
\end{prop}
\begin{proof} By the Gauß formula \cref{e:2.5a}, for any smooth extension $\widetilde{f}$ of $f$ to a neighborhood of~$M$ in $\mathcal{X}$, we have
	\begin{align}
	\nabla_{\alpha} \nabla_{\bba } f
	 = 
	\widetilde{\nabla}_{\alpha} \widetilde{\nabla}_{\bba } \widetilde{f} + \II(Z_{\alpha}, Z_{\bba }) \widetilde{f},
	\end{align}	
	By assumption, we can take $\widetilde{f}$ to be pluriharmonic and thus $\widetilde{\nabla}_{\alpha} \widetilde{\nabla}_{\bba } \widetilde{f}=\partial_\alpha \bar{\partial}_{\beta} \widetilde{f} = 0$. Consequently,
	\begin{align}
	\Box_b f
	& = 
	- h^{\alpha\bba } \nabla_{\alpha} \nabla_{\bba } f \notag \\
	& = 
	- h^{\alpha\bba } \II(Z_{\alpha} , Z_{\bba }) \widetilde{f} \notag \\
	& = - n \overline{\mean} \widetilde{f}.
	\end{align}
	Write $\mean = \mean^{A} \partial_{A}$ in local coordinates. Then
	\begin{equation}
	\Box_b \left(\overline{Z}^A|_M \right) = -n \overline{\mean} \overline{Z}^A = -n \sum_{B} \mean^{\overline{B}} \partial_{\overline{B}} \overline{Z}^A = -n \mean^{\overline{A}},
	\end{equation}
	from which \cref{e:belform} follows.
\end{proof}
\begin{defn}[\cite{dragomir1995pseudohermitian,ebenfelt2004rigidity}]
	Let $(M,\theta)$ and $(N,\eta)$ be strictly pseudoconvex pseudohermitian manifolds and $F \colon (M,\theta) \to (N,\eta)$ a CR immersion. We say that $F$ is a \emph{pseudohermitian immersion} if $F^{\ast}\eta = \theta$ and $F_{\ast} T = T'$, where $T$ and $T'$ are the Reeb fields that correspond to $\theta$ and $\eta$, respectively.
\end{defn}
If $F$ is a pseudohermitian immersion, then the pair of pseudohermitian structures $(\theta, \eta)$ is admissible in the sense of \cite{ebenfelt2004rigidity}.

The following is a CR analogue of the Takahashi theorem \cite{takahashi1966minimal}.

\begin{thm}[Takahashi-type theorem]\label{cor:taka} Let $(M,\theta)$ be a pseudohermitian manifold and let $Z \colon M \to (\mathbb{C}^{N}, \omega:=i\partial\bar{\partial}\|Z\|^2)$ be a semi-isometric CR immersion. Suppose that $\Box_b \overline{Z} = \lambda \overline{Z}$, componentwise, then
	\begin{enumerate}[(i)]
		\item $\lambda > 0$,
		\item $Z(M) \subset r\cdot \mathbb{S}^{2N-1}$, where $r = \sqrt{n/\lambda}$,
		\item $Z \colon M \to r\cdot \mathbb{S}^{2N-1}$ is a pseudohermitian immersion.
	\end{enumerate}
Conversely, if $F \colon M \to r\cdot \mathbb{S}^{2N-1}$ is a pseudohermitian immersion, then $\Box_b \overline{F} = (n/r^2) \overline{F}$.
\end{thm}
\begin{proof}
Let $W = W^A \partial_A$ be a $(1,0)$-vector field. Suppose that $W|_M$ is tangent to $M$, then
\begin{equation}
	\langle T, W \rangle = \langle T, W \rangle_{L_{\theta}} =0.
\end{equation}
On the other hand, since $\mean -\overline{\mean} = iT$, we have
\begin{align}
	0 = -i \langle W , T \rangle 
	=
	\langle W , \overline{\mean} \rangle
	& =
	\sum_{A = 1}^{N} W^A \overline{\mean^{A}} \notag \\
	& =
	\sum_{A = 1}^{N} W^A\left(-\frac{1}{n} \Box_b\left(\overline{Z^A}|_M\right)\right) \notag \\
	& = -\frac{ \lambda}{n} \sum_{A = 1}^{N}W^A \overline{Z^A} \notag \\
	& = -\frac{ \lambda}{n}	W \cdot \left(\|Z\|^2\right).
\end{align}
Thus, $\|Z\|^2$ is a positive constant on $M$ and this proves (ii) for some $r>0$.

Let $\Theta_r$ be the standard pseudohermitian structure on $r\cdot \mathbb{S}^{2N-1}$. Thus, $\varphi:=Z|_M \colon M \to r\cdot \mathbb{S}^{2N-1} $ is a CR immersion of $M$ into the sphere and hence 
\begin{equation}
	e^u \theta = \varphi^{\ast} \Theta_r,
\end{equation}
for some function $u$. Thus, 
\begin{align}
 e^u d\theta|_{H(M)} = d\left(\varphi^{\ast} \Theta_r \right)|_{H(M)} 
	& =
	d(\varphi^{\ast}(\iota^{\ast} \bar{\partial} \|Z\|^2))|_{H(M)} \notag \\
	& =
	(\iota \circ \varphi)^{\ast} (i\partial\bar{\partial} \|Z\|^2 )|_{H(M)} \notag \\
	& =
	d\theta|_{H(M)}.
\end{align}
Hence $u = 0$ and thus $\varphi^{\ast} \Theta_r = \theta$, as desired.

To show that $\varphi$ is pseudohermitian, observe that on $M$, $T = -i (\mean - \bar{\mean})$ coincides with the restriction of the Reeb field of $r \cdot \mathbb{S}^{2N-1}$ on $M$. This proves (iv).

Finally, if (iv) holds, then $|\mean |^2$ coincides with the transverse curvature of the sphere $r\cdot \mathbb{S}^{2N-1}$, i.e., $|\mean |^2 = r^2$, and hence (i) and (iii) follows immediately.	
\end{proof}
\begin{example}
Let $(\mathbb{S}^{2n+1},\Theta)$ be the unit sphere with the standard pseudohermitian structure. For each $q\geq 1$, let $H = (H^1,H^2,\dots,H^N)$ be a CR mapping from $\mathbb{S}^{2n+1} \to \mathbb{C}^N$ whose components $H^j$'s are the restrictions of homogeneous polynomials of degree $q$. If $H$ is a semi-isometric CR immersion, i.e., $d\Theta = H^{\ast}(i\partial\bar{\partial} \|Z\|^2)$ on $\mathbb{S}^{2n+1}$, then \cref{cor:taka} implies that $H(\mathbb{S}^{2n+1}) \subset r \cdot \mathbb{S}^{2N-1}$, with $r = 1/\sqrt{q}$. Moreover, by a result of Rudin-D'Angelo (see \cite[page 159]{d1993several}), $H = U\circ H_q$, where $U$ is unitary and $H_q$ is the map defined by 
\begin{equation}
 	H_q(z) : = \frac{1}{\sqrt{q}}\left( \cdots, \sqrt{\binom{q}{\alpha}}\, z^{\alpha}, \cdots\right), \quad |\alpha| = q.
\end{equation}
We remark that $H_q$ is also minimal as Riemannian immersion of $\mathbb{S}^{2n+1}$ into the sphere $r\cdot \mathbb{S}^{2N-1}$, both are equipped with the standard metric, see \cite{dragomir--tomassini}.
\end{example}
	
\section{Total pseudohermitian tension and an eigenvalue estimate for $\Box_b$}\label{sec:tension}

Let $(M,\theta)$ be a pseudohermitian manifold and $(\mathcal{X},\omega)$ a Kähler manifold with the fundamental $(1,1)$-form $\omega$. Let $f\colon (M,\theta) \to (\mathcal{X},\omega)$ be a $\mathcal{C}^2$ map from $M$ into a Kähler 
manifold $\mathcal{X}$ with the Kähler form $\omega$. In \cite{li--son/proc},
Li and the present author introduced and studied a notion of pseudohermitian harmonic maps. This notion is similar to that of harmonic maps between Kähler manifolds. Namely, we define the $\bar{\partial}_b$-energy functional by
\begin{equation}\label{e:energy}
E[f]:=
\int_M g_{I\bar{J}} f^I_{\aba} f^{\bar{J}}_{\beta} h^{\beta\aba} d\vol_{\theta}.
\end{equation}
Here, $Z^J$ is the local coordinates on $\mathcal{X}$, $f^J:=Z^J \circ f$, and $g_{I\bar{J}}$ is the K\"ahler metric tensor.
The Greek indices indicate the derivative along CR and anti-CR directions. Clearly, $E[f]$ is real-valued and nonnegative, $E[f] = 0$ if and only if $f$ is a CR map. A critical point of $E[\cdot]$ is 
called a \textit{pseudohermitian harmonic} map in \cite{li--son/proc}. The associated Euler-Lagrange equation for $E[f]$ is given by the vanishing of a $(1,0)$-vector field which is called \textit{pseudohermitian tension field}. Namely, for a $\mathcal{C}^2$ map $f$,
the tension field $\tau[f]$ is the $(1,0)$-vector field along $f(M)$ given by
\begin{equation}
\tau[f]
:=
h^{\beta\aba} \left( f^{I}_{\aba, \beta} + \Gamma^I_{JK} f^J_{\aba} f^{K}_{\beta} \right) e_I, \quad e_I: = \partial/\partial z^I.
\end{equation}
%Here a Greek index preceded by a comma indicates the Tanaka-Webster covariant derivative.
Then, $f$ is {pseudohermitian harmonic} iff $\tau[f] = 0$.

The following generalizes the Reilly-type estimate for the first positive eigenvalue of the Kohn Laplacian. Its proof use similar ideas as in \cite{li--lin--son}.

\begin{prop}\label{prop:tensionei}
Let $(M^{2n+1},\theta)$ be a compact strictly pseudoconvex pseudohermitian manifold and $f \colon M\to (\mathbb{C}^d, \omega)$, $\omega:=i\partial\bar{\partial} \|Z\|^2$, a $\mathcal{C}^2$-map. 
If $f$ is \emph{not} a CR map, then $E[f] \ne 0$ and
		\begin{equation}\label{e:est1}
			\lambda_1
			\leq 
			\frac{1}{E[f]}\int_M \left|\tau[f]\right|^2_{\omega} d\vol_{\theta}.
		\end{equation}
If the equality occurs, then for each $I=1,2,\dots , d$, $\tau(f)^I$ is an eigenfunction corresponding to $\lambda_1$, provided that it is not identically zero.
\end{prop}
The integral in the right-hand side of \cref{e:est1} is called the \textit{total pseudohermitian tension} of the mapping.
\begin{proof}
 In the standard coordinates on $\mathbb{C}^n$, $\Gamma^{I}_{JK}=0$. Thus,
	\begin{equation}
	\tau(f)^{I}
	=
	h^{\beta\aba} f^{I}_{\aba, \beta}
	=
	-\Box_b f ^{I}.
	\end{equation}
	By the usual variational characterization of $\lambda_1$,
	\begin{align}\label{e:var1}
	\lambda_1 \int_M |\bar{\partial}_b f^I|^2 d\vol_{\theta}
	\leq
	\int_M |\Box_b f^I|^2 \, d\vol_{\theta}
	=
	\int_M |\tau(f)^I|^2 d\vol_{\theta}.
	\end{align}
	By \cref{e:energy} and the fact that $g_{I\bar{J}} = \delta_{IJ}$, we have
	\begin{equation}
	E[f] = \sum_{I=1}^N \int_M f^{I}_{\aba} \overline{f^{I}_{\bba}} h^{\beta\aba} d\vol_{\theta}
	=
	\sum_{I=1}^N \int_M |\bar{\partial}_b f^I|^2 d\vol_{\theta}.
	\end{equation}
	Summing over $I$ the equation \cref{e:var1}, we obtain
	\begin{equation}
	\lambda_1 E[f]
	=
	\lambda_1 \sum_{I=1}^d \int_M |\bar{\partial}_b f^I|^2 d\vol_{\theta}
	\leq \int_M \left|\tau(f)\right|^2_{\omega} d\vol_{\theta},
	\end{equation}
	as desired.	The characterization of equality is immediate from \cref{e:var1}.
\end{proof}

In the following we give a nontrivial example for which the pseudohermitian tension can be computed explicitly.
\begin{example}\rm 
	Let $M$ be the compact strictly pseudoconvex Reinhardt real hypersurface in $\mathbb{C}^{n+1}$ defined by $\rho = 0$, where
	\begin{equation}
	\rho
	=
	\sum_{j=1}^{n+1} \left( \log |z_j|^2\right)^2 - 1.
	\end{equation}
	Let $\theta:= i\bar{\partial} \rho|_{M}$ and $f^j(z) = \log |z_j|^2$. Consider $f = (f^j \colon j=1,2,\dots , n+1) \colon M \to \mathbb{C}^{n+1}$. Observe that 
	\begin{equation}
	|\bar{\partial}_b \log |z_j|^2 |^2 = \frac{1}{2} - \frac{1}{2} (\log |z_j|^2)^2.
	\end{equation}
	Thus,
	\begin{equation}
	\sum_{j=1}^{n+1} |\bar{\partial}_b \log |z_j|^2 |^2
	=
	\sum_{j=1}^{n+1} \left( \frac{1}{2} - \frac{1}{2} (\log |z_j|^2)^2 \right)
	=
	\frac{n+1}{2} - \frac{1}{2}\sum_{j=1}^{n+1} (\log |z_j|^2)^2
	=
	\frac{n}{2}.
	\end{equation}
	This implies that the $\bar{\partial}_b$-energy of $f$ is 
	\begin{equation}
	E[f] = \int_M \left(\sum_{j=1}^{n+1} |\bar{\partial}_b f^j |^2 \right) d\vol_{\theta}
	=
	\frac{n}{2} \vol(M,\theta).
	\end{equation}
	Since $\log |z_j|^2$ are pluriharmonic, from \cref{e:klpluriharmonic} we obtain
	\begin{equation}
	\tau[f]^j = - \Box_b (\log |z_j|^2) = - \frac{n}{2} \log |z_j|^2.
	\end{equation}
	Therefore, $\lambda = n/2$ is an eigenvalue of $\Box_b$ and hence $\lambda_1 \leq n/2$.
	On the other hand, the total tension can be computed as follows.
	\begin{equation}
	|\tau[f]|^2
	=
	\sum_{j=1}^{n+1} |\tau[f]^j|^2
	=
	\frac{n^2}{4} \left( \log |z_j|^2 \right)^2
	=
	\frac{n^2}{4}.
	\end{equation}
	Thus, we have
	\begin{equation}
	\lambda_1 \leq \frac{n}{E[f]}\int_M |\tau[f]|^2 d\vol_{\theta} = \frac{n}{2},
	\end{equation}
	as expected.
\end{example}
\begin{proof}[Proof of Theorem~\ref{thm:1}] 
	In view of \cref{prop:beltrami}, we shall compute the $\bar{\partial}_b$-energy and the total tension of the conjugate of the CR immersion $F$. The argument below is almost the same as in \cite{li--son} and similar to usual proofs for Riemannian case. Precisely,
	\begin{equation}
	\int_M |\tau(\bar{F})|^2_\omega
	=
	\int_M \sum_{J=1}^{N} | \Box_b \bar{F}^J |^2 = n^2\int_M |\mean |_{\omega}^2.
	\end{equation}
	In local computations, we can identify $M$ with its image $F(M)$ and write
	\begin{equation}
	Z_\alpha 
	=
	\mu_{\alpha}^I \partial_{I} .
	\end{equation}
	Here $I$ runs from $1$ to $N$.	Thus, $h_{\alpha\bba }
	=
	\langle Z_{\alpha}, Z_{\bba }\rangle 
	=
	\delta_{IJ} \mu_{\alpha}^I	\mu_{\bba }^{\bar{J}} 
	$.	On the other hand, $Z_\alpha z^I = \mu_{\alpha}^I$, and thus
	\begin{equation}
	\sum_{I=1}^{N}(Z_\alpha z^I)(Z_{\bba } \zba^I) = h_{\alpha\bba }.
	\end{equation}
	Here $z^I$'s are the coordinates in $\mathbb{C}^N$.	Consequently,
	\begin{equation}
	\sum_{I=1}^{N}|\bar{\partial}_b \zba^I|^2 = n,
	\end{equation}
	and therefore 
	\begin{equation}
		E[\bar{z}]
		=
	\int \sum_{I=1}^{N}|\bar{\partial}_b \zba^I|^2 = n\vol(M).
	\end{equation}
	Thus, the estimate \cref{e:est0} follows from \cref{prop:tensionei}.
	
	Suppose that the equality occurs. Put $b^I = \Box_b \zba^I$. Then clearly $\Box_b b^I = \lambda_1 b^I$.	This implies that $\Box_b(b^I -\lambda_1 \bar{z}^I) =0$ and hence there are globally defined CR functions
	$\varphi^I$'s such that $b^I = \lambda_1 \bar{z}^I + \varphi^I$. In particular, $b^I$'s are (complex-valued) CR-pluriharmonic eigenfunctions that correspond to~$\lambda_1$. The proof is complete.
\end{proof}

\section{The linearity of totally pseudohermitian umbilic CR immersions}\label{sec:6}
In this section, we prove \cref{thm:2}, \cref{cor:linearity}, and \cref{cor:3dim}. We first prove the following lemma.
\begin{lem}\label{lem:sp}
	Let $\iota \colon (M, \theta) \hookrightarrow (\mathbb{C}^{N},\omega)$, $\omega: = i\partial\bar{\partial}\|Z\|^2$, be a pseudohermitian submanifold, $\iota^{\ast}\omega = d\theta$. If $\II^{\circ} = 0$, then $M$ is CR spherical. Moreover, the pseudohermitian Ricci curvature and mean curvature are constant and satisfy
	\begin{align}\label{e:a}
	 \Ric &= (n+1)|\mean |^2 d\theta|_{H(M)},\\
	 \quad 
	 R &= n(n+1) |\mean |^2.
	\end{align}
	Furthermore, the pseudohermitian torsion vanishes, i.e., $A_{\alpha\beta} = 0$.
\end{lem}
\begin{proof} 
	By assumption, $\omega_{\alpha\gamma}^a=0$ and the Gauß equation \cref{e:gauss} implies that the Chern--Moser tensor of $M$ vanishes identically. Thus $M$ is locally spherical, provided that $n\geq 2$. Equations \cref{e:a} follows by tracing the Gauß equation. Moreover, $A_{\alpha\beta} = 0$, by \cref{e:gausstorsion}. If $n\geq 2$, then it follows from a Bianchi identity (Eq. (2.11) in \cite{lee1988pseudo}) that scalar curvature $R$ must be a constant and therefore, the mean curvature $|\mean|^2$ is a positive constant.
	
	For the case $n=1$, the arguments above do not work. For this case, we argue as follows. From \cref{prop:codazzi-mainardi}, we deduce that
	\begin{align}\label{e:63}
	0 = - (D_{\overline{Y}} \II) (X, Z) + \langle \overline{Y}, Z\rangle D_{X} H + 
	\langle X, \overline{Y}\rangle D_ZH.
	\end{align}
	By assumption, $\II(U,V) = 0$ for all $U,V \in T^{1,0}M$ and hence $(D_{\overline{Y}} \II) (X, Z) = 0$. Then \cref{e:63} reduces to
	\begin{align}
	0 = \langle \overline{Y} , Z \rangle D_X H + \langle X, \overline{Y}\rangle D_ZH.
	\end{align}
	Hence, $D_ZH = 0$ for all (1,0) tangent vector $Z$. By \cref{prop:constantmean}, $|H|^2$ is a constant.	To show that $M$ is CR spherical, observe that $A_{11} = 0$ and hence, by Lemma 2.2 of \cite{cheng1990burns} (cf. \cite{ebenfelt2017umbilical}), the Cartan's 6th-order tensor vanishes:
	\begin{equation}
		Q_{1}{}^{\bar{1}} = \frac{1}{6}R_{,1}{}^{\bar{1}} = 0,
	\end{equation}
	since $R = 2|H|^2$ is also a constant. Hence $M$ is locally CR spherical by Cartan's theorem. The proof is complete.
\end{proof}
\begin{proof}[Proof of \cref{thm:2}] 
By \cref{lem:sp}, the Ricci tensor $R_{\alpha\bba}$ has a positive lower bound and the torsion $A_{\alpha\beta} = 0$. Therefore, $M$ must be compact. The argument uses Meyer's theorem and is well-known to the folklore: If the torsion vanishes and $R_{\alpha\bba}$ has a positive lower bound, then the Ricci curvature of the Levi-Civita connection associated to some adapted Webster metric $g_{\theta} : = G_{\theta} + \epsilon \theta \odot \theta$ has a positive lower bound for some positive constant $\epsilon$ (see e.g. Proposition~8 in \cite{wang2015remarkable}) and hence $M$ must be compact (and has finite fundamental group) by Meyer's theorem \cite[Theorem 3.85]{gallot--hulin--lafontaine}.
	
Since $R_{\alpha\bba} = (n+1)|H|^2h_{\alpha\bba}$, with $|H|^2$ is constant, the lower bound for the first positive eigenvalue of Chanillo-Chiu-Yang \cite{chanillo--chiu--yang} (the case $n\geq 2$ is due to Chang-Wu; see \cite{li--son--wang}) reads $\lambda_1 \geq n|\mean|^2$. This bound also holds in three-dimensional case since $A_{\alpha\beta} = 0$ (The lower bound of \cite{chanillo--chiu--yang} requires that the so-called CR Paneitz operator is nonnegative. This condition is fulfilled when the pseudohermitian torsion vanishes identically.) On the other hand, the upper bound in \cref{thm:1} reads	
\begin{equation}
	\lambda_1 \leq \frac{n}{\vol(M)} \int |\mean|^2 = n |\mean|^2,
\end{equation}
also because $|\mean|^2$ is constant on $M$. Thus,
\begin{equation}
	\lambda_1 = n|\mean|^2.
\end{equation}
By the characterization of the CR sphere in \cite{li--son--wang}, $(M^{2n+1},\theta)$ must be globally CR equivalent to the standard CR sphere. (In fact, we do not need the result in \cite{li--son--wang} in its full generality as we already know that $A_{\alpha\beta}=0$. Under this condition, the characterization of the sphere in \cite{li--son--wang} also holds also in three-dimensional case.) Moreover, by the characterization of the equality in \cref{thm:2}, each function $b^{I} : = \Box_b \overline{F^{I}} $ is either a constant or an eigenfunction corresponding~$\lambda_1$. 

We can now assume that $M = \mathbb{S}^{2n+1} \subset \mathbb{C}^{n+1}$ is given by $M = \{\|z\|^2 = 1\}$. It is well-known that the eigenspace of $\Box_b$ corresponding to the first positive eigenvalue is spanned by the restrictions to $M$ of $\bar{z}^j$, $j = 1,2,\dots , n+1$ (i.e., the restrictions of the homogeneous harmonic polynomials of bidegree $(0,1)$). Since $b^I$ is an eigenfunction or a constant, there exist constants $c_1,\dots , c_{n+1}$ such that $b^{I} = \sum_{j=1}^{n+1} c_j \bar{z}^j|_{M}$ and hence
\begin{equation}
 \sum_{j=1}^{n+1} \bar{c}_j z_j|_{M}
 =
	\overline{b}^I = \overline{\Box}_b F^{I} = -n (\mean \cdot F^{I})|_{M}
	=
	-n(z^j \partial_j F^{I})|_{M}
\end{equation}
Here we use the fact that $F^{I}$ is CR and thus the Beltrami-type formula \cref{e:klpluriharmonic} can be applied for $\overline{F}^{I}$. By the well-known CR extension theorem, $F^{I}$ holomorphically extends to the unit ball and the identity $\sum_{j=1}^{n+1} \bar{c}_j z_j = nz^j \partial_j F^{I}$ holds on the unit ball. We can conclude (by considering the power series expansion at the origin) that $F^{I}$ is either a constant (when all $c_j = 0$) or a linear function. Thus, $F$ is an linear embedding.
\end{proof}
\begin{proof}[Proof of \cref{cor:3dim}]
	Let $\iota \colon \mathbb{S}^{2N-1} \hookrightarrow \mathbb{C}^N$ be the standard inclusion and let $F:= \iota \circ \phi$. Then $F\colon M \to \mathbb{C}^N$ is a semi-isometric CR immersion. By assumptions, $\II^{CR}_M(\phi) = 0$ and hence $F$ is totally pseudohermitian umbilic by \cref{prop:2sff}. Thus, $M$ is CR spherical by \cref{lem:sp}. Locally, there exists a local CR diffeomorphism $\varphi \colon \mathbb{S}^{2n+1} \to M^{2n+1}$. Put $G: = \phi \circ \varphi$, then $G$ extends to a global CR immersion $\widetilde{G}$ from $\mathbb{S}^{2n+1}$ into $\mathbb{S}^{2N-1}$, which is the restriction of a rational map with poles off $\mathbb{S}^{2n+1}$, by F.~Forstneri\v{c}'s theorem. The extension $\widetilde{G}$ also satisfies $\II_{G(\mathbb{S}^{2n+1})}^{CR} = 0$ globally, by rationality. In particular, \cref{thm:2} applies to $\iota \circ \widetilde{G}$ and gives the desired linearity. The proof is complete.
\end{proof}	
\begin{proof}[Proof of \cref{cor:linearity}] When $M$ is CR spherical and $\mathcal{X}$ is flat, then \cref{prop:2um} implies that traceless component $\II^{\circ}$ vanishes identically and the conclusion follows from \cref{thm:2}.
\end{proof}
\section{An example: the complex Whitney map}
We use the following formula to simplify our calculations for maps between spheres.
\begin{lem}\label{prop:fun}
	Let $M$ be a strictly pseudoconvex real hypersurface defined by $\rho = 0$, with $d\rho\ne 0$ on $M$. Let $\sigma$ be a smooth function in a neighborhood of $M$ and $\hat{\rho} = e^{\sigma} \rho$. Then
	\begin{align}\label{e:tranconf}
	e^{\sigma} r(\hat{\rho})
	=
	r(\rho) + 2\Re (\xi)\, \sigma - |\bar{\partial}_b \sigma|^2.
	\end{align}
	where $\xi$ is the transverse vector field of $\rho$ defined as in \cref{e:tvdef} and the norm of $|\bar{\partial}_b \sigma|^2$ is in terms of $\theta: = \iota^{\ast}(i\bar{\partial}\rho)$.
\end{lem}
\begin{proof} 
Since $J(\rho) \ne 0$, the matrix $\psi_{j\kba}:= \rho_{j\kba} + (1-r)\rho_j \rho_{\kba}$ is invertible; cf. \cite{lee--melrose} (see also \cite{li--son}). Let $\psi^{\kba j}$ be its inverse, $h^{\kba j}
=
\psi^{\kba j} - \xi^{\kba} \xi^{j}
$, and $\hat{\xi}^j = \xi^j - \sigma_{\bar{k}} h^{j\bar{k}}$. We can check that $\hat{\xi}$ satisfies the defining properties of the transverse vector field $\xi$ in \cref{e:tvdef} and thus \cref{e:tranconf} follows.
\end{proof}
\begin{example}\label{ex:whitney}
The complex Whitney map $\Wm_n$ is a quadratic polynomial map which restricts to a CR embedding of $\mathbb{S}^{2n+1}$ into $\mathbb{S}^{4n+1}$. Precisely (see \cite[Chapter~5]{d1993several})
\begin{align}
\Wm(z_1,z_2,\dots , z_n ,w)
=
(z_1, \dots , z_n , z_1w , \dots, z_n w, w^2).
\end{align}
Let $\theta : = (1+|w|^2) \Theta$, where $\Theta$ is the standard pseudohermitian structure on $\mathbb{S}^{2n+1}$. Then $\Wm \colon (\mathbb{S}^{2n+1}, \theta) \to \mathbb{C}^{2n+1}$ is a semi-isometric CR immersion.

We claim that $p \in \mathbb{S}^{2n+1}$ is an umbilical point of $\Wm$ if and only if $p = (0,\dots, 0 , e^{it})$, $t$ is real. 
Indeed, let $\sigma = \log (1+|w|^2)$ so that $\theta = e^{\sigma}\Theta$. By Lee's formula for the Webster scalar curvature \cite{lee1988pseudo},
\begin{align}
	R_{\theta}
	=
	e^{-\sigma} \left( R_{\Theta} + (n+1) \Delta_b \sigma - n(n+1)|\bar{\partial}_b \sigma|^2 \right), \quad R_{\Theta} = n(n+1).
\end{align}
On the other hand, let $\hat{\rho}: = e^{\sigma} (\|Z\|^2 - 1)$, then by \cref{prop:transmean}, 
\begin{align}
	|H|^2 \circ \Wm
	& =
	r(\hat{\rho}) \notag \\
	& =
	e^{-\sigma} \left( r(\|Z\|^2 - 1 ) + 2\Re (\xi)\, \sigma - |\bar{\partial}_b \sigma|^2 \right) \notag \\
	& =
	e^{-\sigma}\left(1 + 2\Re (\xi)\, \sigma - |\bar{\partial}_b \sigma|^2 \right),
\end{align}
where $\xi = z^{j} \partial_j$, and $|\bar{\partial}_b \sigma|^2$ is computed with respect to the standard pseudohermitian structure. Thus,
\begin{align}
	e^{\sigma}|\II^{CR}|^2
	=
	n(n+1) |H|^2 - R_{\theta}
	=
	(n+1)(2n\Re (\xi)\,\sigma - \Delta_b \sigma).
\end{align}
On the sphere with $\rho = \|Z\|^2 -1$, we have that,
\begin{align}
	(n+1)(2n\Re (\xi)\,\sigma - \Delta_b \sigma)
	& =
	2(n+1) (\delta_{jk} - z^{j} \zba^{k}) \sigma_{j\kba} \notag \\
	& =
	2(n+1)\frac{(1-|w|^2)^2}{(1+|w|^2)^2}.
\end{align}
Therefore, $\II^{CR}(p)$ vanishes iff $|w| = 1$ and hence the claim follows.

This example shows that the condition $N\leq 2n$ in \cref{thm:2,prop:2um,cor:2um,thm:umbilichypersurface} is necessary.
\end{example}
\bigskip

\textbf{Acknowledgment.} The author thanks anonymous referees for pointing out and addressing some delicate issues in a previous version of Proposition 2.15 and its proof, and for useful comments.

\end{document}